\documentclass[11pt,amssymb,reqno]{amsart}
\usepackage{graphicx}
\usepackage{amsmath,amssymb,amsfonts,euscript,enumerate}
\usepackage{amsthm}
\usepackage{color,wrapfig}
\usepackage{pxfonts}
\usepackage{verbatim}
\newtheorem{thm}{Theorem}[section]
\newtheorem{cor}[thm]{Corollary}
\newtheorem{lemma}[thm]{Lemma}
\newtheorem{prop}[thm]{Proposition}

\newcommand{\R}{{\mathbb{R}}}
\newcommand{\Z}{{\mathbb{Z}}}

\newcommand{\La}{\triangle}
\newcommand{\bs}{\backslash}

\newcommand{\1}{\partial}
\newcommand{\2}{\overline}
\newcommand{\3}{\varepsilon}

\begin{document}

\title{Existence of Neumann and singular solutions of the fast diffusion equation}

\author[Kin Ming Hui]{Kin Ming Hui}
\address{Kin Ming Hui:
Institute of Mathematics, Academia Sinica,\\
Taipei, 10617, Taiwan, R.O.C.}
\email{kmhui@gate.sinica.edu.tw}

\author[Sunghoon Kim]{Sunghoon Kim}
\address{Sunghoon Kim:
Department of Mathematics, School of Natural Sciences, The Catholic University of Korea,
43 Jibong-ro, Wonmi-gu, Bucheon-si, Gyeonggi-do, 420-743, Republic of Korea}
\email{math.s.kim@catholic.ac.kr}


\begin{abstract}
Let $\Omega$ be a smooth bounded domain in $\R^n$, $n\ge 3$, $0<m\le\frac{n-2}{n}$,
$a_1,a_2,\dots, a_{i_0}\in\Omega$, $\delta_0=\min_{1\le i\le i_0}\mbox{dist }(a_i,\1\Omega)$ and
let $\Omega_{\delta}=\Omega\setminus\cup_{i=1}^{i_0}B_{\delta}(a_i)$ and $\hat{\Omega}=\Omega\setminus\{a_1\,\dots,a_{i_0}\}$.
For any $0<\delta<\delta_0$ we will prove the existence and uniqueness of positive solution of the Neumann problem for the equation $u_t=\Delta u^m$
in $\Omega_{\delta}\times (0,T)$ for some $T>0$. We will prove the existence of singular solutions of this equation in $\hat{\Omega}\times (0,T)$ for some $T>0$ that blow-up at the points $a_1,\dots, a_{i_0}$. 
\end{abstract}

\maketitle
\vskip 0.2truein

\setcounter{equation}{0}
\setcounter{section}{0}

\section{Introduction}
\setcounter{equation}{0}
\setcounter{thm}{0}

Recently there is a lot of research on the equation
\begin{equation}\label{fast-diffusion-eqn}
u_t=\Delta u^m
\end{equation}
by M.~Bonforte, E. Chasseigne, M.~Fila, G. Grillo, J.L.~Vazquez, M.~Winkler, E.~Yanagida \cite{BGV1}, \cite{BGV2}, \cite{BV1}, \cite{BV2}, \cite{BV3}, \cite{CV}, \cite{FVWY},
P.~Daskalopoulos, M.Del~Pino and N.~Sesum \cite{DPS}, \cite{DS1}, \cite{DS2}, S.Y.~Hsu [Hs2-3], K.M.~Hui [Hu2-3],
M.Del~Pino and M.~S\'aez \cite{PS}, L.A.~Peletier and H.~Zhang \cite{PZ}, etc. This equation arises in many physical models.
When $m>1$, it is called the porous medium which models the diffusion of gases through porous media \cite{A}. When
$m=1$, \eqref{fast-diffusion-eqn} is the heat equation. When $0<m<1$, it is usually called the fast diffusion equation.
When $m=\frac{n-2}{n+2}$, \eqref{fast-diffusion-eqn} appears in the study of Yamabe flow on $\R^n$. In fact
the metric $g_{ij}=u^{\frac{4}{n+2}}dx^2$, $u>0$, is a solution of the Yamabe flow  \cite{DS2}, \cite{PS},
\begin{equation*}
\frac{\1 g_{ij}}{\1 t}=-Rg_{ij}
\end{equation*}
in $\R^n$, $n\ge 3$, if and only if $u$ is a solution of
\begin{equation*}
u_t=\frac{n-1}{m}\Delta u^m
\end{equation*}
with $m=\frac{n-2}{n+2}$ where $R$ is the scalar curvature of $g_{ij}$. We refer the readers to the book
\cite{V3} by J.L.~Vazquez for the basics of \eqref{fast-diffusion-eqn} and the books \cite{DK}, \cite{V2},
by P.~Daskalopoulos, C.E.~Kenig and J.L.~Vazquez for the most recent results of \eqref{fast-diffusion-eqn}. We also refer to the paper \cite{BV3}, by M. Bonforte and J. L. Vazquez for the non local version of \eqref{fast-diffusion-eqn}.

As observed by L.~Peletier \cite{P} and J.L~Vazquez \cite{V1} there is a big difference on the behaviour of solutions
of \eqref{fast-diffusion-eqn} for $(n-2)/n<m<1$, $n\ge 3$,
and for $0<m\le (n-2)/n$, $n\ge 3$. For example there is a $L^1-L^{\infty}$ regularizing effect for the solutions
of
\begin{equation}\label{Cauchy-problem}
\left\{\begin{aligned}
u_t=&\Delta u^m, u\ge 0,\quad\mbox{ in }\R^n\times (0,T)\\
u(x,0)=&u_0\qquad\qquad\quad\mbox{ in }\R^n
\end{aligned}\right.
\end{equation}
with $0\le u_0\in L^1_{loc}(\R^n)$ for any $(n-2)/n<m<1$ \cite{HP}, \cite{DaK}. However there is no such $L^1-L^{\infty}$ regularizing
effect for  solutions of \eqref{Cauchy-problem} when $0<m\le (n-2)/n$, $n\ge 3$, \cite{V2}.
When $\frac{n-2}{n}<m<1$, existence and uniqueness of global weak solution of \eqref{Cauchy-problem}
for any $0\le u_0\in L^1_{loc}(\R^n)$ has been proved by M.A.~Herrero and M.~Pierre in \cite{HP}. When $0<m\le (n-2)/n$ and $n\ge 3$,
existence of positive smooth solutions of \eqref{Cauchy-problem} for any $0\le u_0\in L_{loc}^p(\R^n)$, $p>\max(1,(1-m)n/2)$, satisfying the condition,
\begin{equation}\label{initial-value-average-lower-bd-0}
\liminf_{R\to\infty}\frac{1}{R^{n-\frac{2}{1-m}}}\int_{|x|\le R}u_0\,dx
\ge C_1T^{\frac{1}{1-m}}
\end{equation}
for some constant $C_1>0$ is proved by S.Y.~Hsu in \cite{Hs3}. 

In this paper we will study the existence of singular solutions of 
\eqref{fast-diffusion-eqn}. Study of singular solutions of nonlinear elliptic equations were also obtained by H.~Brezis and L.~Veron \cite{BrV}, B.~Gidas and J.~Spruck \cite{GS}, etc. In order to study the singular solutions of \eqref{fast-diffusion-eqn}
we will first prove the existence of positive smooth solution of the Neumann problem for \eqref{fast-diffusion-eqn} in smooth bounded domains with a finite numbers of holes when $0<m\le (n-2)/n$, $n\ge 3$. When $n\ge 3$ and $m=(n-2)/n$, we will prove the existence of  singular solutions of \eqref{fast-diffusion-eqn} in a smooth bounded domain
that blow-up at a finite number of points in the domain. More precisely let $\Omega$ be a smooth bounded domain in $\R^n$, $n\ge 1$, $a_1, a_2, \dots, a_{i_0}\in\Omega$, $\delta_0=\min_{1\le i,j\le i_0}\left(\mbox{dist }(a_i,\1\Omega),|a_i-a_j|\right)/2$,
$\Omega_{\delta}=\Omega\setminus\cup_{i=1}^{i_0}B_{\delta}(a_i)$, $\hat{\Omega}=\Omega\setminus\{a_1,\dots,a_{i_0}\}$ and $\hat{\R}=\R\setminus\{a_1,\dots,a_{i_0}\}$. We will prove the following three main theorems.

\begin{thm}\label{Neumann-holes-problem-existence-thm}
Let $n\ge 3$, $0<m\le\frac{n-2}{n}$, $0<\delta<\delta_0$, $0\le u_0\in L^p(\Omega_{\delta})$ for some constant $p>\frac{n(1-m)}{2}$, $0\le f\in L_{loc}^{\infty}(\1\Omega\times [0,\infty))$ and $0\le g_i\in L_{loc}^{\infty}(\1 B_{\delta}(a_i)\times [0,\infty))$ for all $i=1,\cdots,i_0$. Suppose either $u_0\not\equiv 0$ on $\Omega_{\delta}$ or
\begin{equation*}\label{f-g-L1-positive-condition}
\int_0^t\int_{\partial\Omega}f\,d\sigma ds
+\sum_{i=1}^{i_0}\int_0^t\int_{\partial B_{\delta}(a_i)}g_i\,d\sigma ds>0\quad\forall t>0.
\end{equation*}
Then there exists a unique solution $u$ for the equation
\begin{equation}\label{Neumann-holes-eqn}
\begin{cases}
\begin{aligned}
u_t=&\La u^m\qquad\,\,\,\mbox{ in }\Omega_{\delta}\times(0,\infty)\\
\frac{\partial u^m}{\partial\nu}&=f\qquad\,\,\,\mbox{ on }\partial\Omega\times(0,\infty)\\
\frac{\partial u^m}{\partial\nu}&=g_i\qquad\,\,\mbox{ on $(\bigcup_{i=1}^{i_0}\partial B_{\delta}(a_i))\times(0,\infty)$}
\quad\forall i=1,\dots,i_0\\
u(x,0)&=u_0(x)\quad\,\mbox{in }\Omega_{\delta}
\end{aligned}
\end{cases}
\end{equation}
that satisfies
\begin{equation}\label{mass-formulua-holes}
\int_{\Omega_{\delta}}u(x,t)\,dx=\int_{\Omega_{\delta}}u_0\,dx+\int_0^t\int_{\partial\Omega}f\,d\sigma ds
+\sum_{i=1}^{i_0}\int_0^t\int_{\partial B_{\delta}(a_i)}g_i\,d\sigma ds
\quad \forall t>0
\end{equation}
where $\frac{\partial}{\partial\nu}$ is the derivative on $\partial\Omega_{\delta}$ with respect to the unit outward normal of the domain 
$\Omega_{\delta}$
Moreover if $f\equiv 0$ on $\1\Omega\times (0,\infty)$ and $g_i$, $i=1,\dots, i_0$, are nonnegative monotone decreasing functions of $t>0$, then  
\begin{equation}\label{Aronson-Benilan-ineqn}
u_t\le\frac{u}{(1-m)t}
\end{equation}
in $\Omega_{\delta}\times(0,\infty)$.
\end{thm}

\begin{thm}\label{singular-soln-existence-thm}
Let $n\ge 3$, $0<m\le\frac{n-2}{n}$, $p>\frac{n(1-m)}{2}$, $0\le f\in L_{loc}^{\infty}(\1\Omega\times [0,\infty))$. Let $0\le u_0\in L_{loc}^p(\hat{\Omega})$ be such that 
\begin{equation}\label{u0-blow-up-condition2}
\frac{C_1}{|x-a_i|^qe^{\frac{1}{\delta_1^2-|x-a_i|^2}}}\le u_0(x)\le\frac{C_2}{|x-a_i|^q}\qquad \forall 0<|x-a_i|\le \delta_1, i=1,\cdots,i_0
\end{equation} 
for some constants $C_1>0$, $C_2>0$, $q\ge\max\left(\frac{n}{2m},\frac{n-2}{m}\right)$ and 
$0<\delta_1<\min\left(\frac{(1-m)q}{4+(1-m)q},\delta_0\right)$. 
Then there exists a solution $u$ of
\begin{equation}\label{singular-Neumann-problem}
\left\{\begin{aligned}
u_t=&\Delta u^m\quad\mbox{ in }\hat{\Omega}\times(0,\infty)\\
\frac{\partial u^m}{\partial\nu}=&f\qquad\, \mbox{ on }\partial\Omega\times(0,\infty)\\
u(x,0)=&u_0(x)\quad\mbox{ in }\hat{\Omega}
\end{aligned}\right.
\end{equation}
such that 
\begin{equation}\label{singular-soln-upper-lower-bd}
u(x,t)\ge\frac{C_1}{|x-a_i|^qe^{\frac{1}{\delta_1^2-|x-a_i|^2}}}\quad\forall 0<|x-a_i|<\delta_1,t>0, i=1,\cdots,i_0
\end{equation}
and
\begin{equation}\label{singular-soln-upper-lower-bd2}
u(x,t)\le\frac{C_T}{|x|^q}\qquad\qquad\forall 0<|x-a_i|\le\frac{\delta_1}{2},0<t\le T, i=1,\cdots,i_0
\end{equation}
hold for some constant $C_T>0$ where $\1/\1\nu$ is the derivative with respect to the unit outward normal on $\1\Omega$.
\end{thm}

\begin{thm}\label{singular-soln-existence-thm2}
Let $n\ge 3$, $0<m\le\frac{n-2}{n}$, $p>\frac{n(1-m)}{2}$. Let $0\le u_0\in L_{loc}^p(\hat{\R})$ be such that 
\eqref{u0-blow-up-condition2} holds for some constants $C_1>0$, $C_2>0$, $q\ge\max\left(\frac{n}{2m},\frac{n-2}{m}\right)$ and 
$$
0<\delta_1<\min\left(\frac{(1-m)q}{4+(1-m)q},\frac{1}{2}\min_{1\le i,j\le i_0}|a_i-a_j|\right).
$$ 
Then there exists a solution $u$ of
\begin{equation}\label{singular-Neumann-problem2}
\left\{\begin{aligned}
u_t=&\Delta u^m\quad\,\,\mbox{ in }\hat{\R}\times(0,\infty)\\
u(x,0)=&u_0(x)\quad\mbox{ in }\hat{\R}
\end{aligned}\right.
\end{equation}
such that \eqref{singular-soln-upper-lower-bd} and \eqref{singular-soln-upper-lower-bd2} hold for some constant $C_T>0$.
\end{thm}

The plan of the paper is as follows. In section two we will prove some a priori estimates for $C^{2,1}$ solution of \eqref{Neumann-holes-eqn}. In section three we will prove Theorem \ref{Neumann-holes-problem-existence-thm}. In section four we will prove Theorem \ref{singular-soln-existence-thm} and Theorem \ref{singular-soln-existence-thm2}.

We start with some notations and definitions that will be used in this paper. Let $\Omega_1\subset\R^n$ be a smooth bounded domain and let $\Sigma_1$, $\Sigma_2$ be relatively open subsets of $\partial\Omega_1$ such that $\partial\Omega_1=\Sigma_1\cup\Sigma_2$ and if $n\geq 2$, then $\overline{\Sigma}_1\cap\overline{\Sigma}_2$ is a $C^2$ manifold of dimension $n-2$. For any $0<m<1$, $0\le u_0\in L^1(\Omega_1)$,
$f\in L_{loc}^1(\Sigma_1\times [0,\infty))$ and $g\in L_{loc}^1(\Sigma_2\times [0,\infty))$, we say that $u$ is a very weak solution (subsolution, supersolution respectively) of
\begin{equation}\label{Dirichlet-Neumann-problem}
\left\{\begin{aligned}
u_t=&\La u^m\quad\mbox{in }\Omega_1\times(0,T)\\
\frac{\partial u^m}{\partial\nu}=&f\qquad\mbox{ on }\Sigma_1\times(0,T)\\
u=&g\qquad\mbox{ on }\Sigma_2\times(0,T)\\
u(x,0)=&u_0(x)\quad\mbox{in }\Omega_1
\end{aligned}\right.
\end{equation}
if $0\le u\in C([0,T);L^1(\Omega_1))$ satisfies ($\ge$, $\le$ respectively)
\begin{align}\label{weak-soln-defn}
&\int_{t_1}^{t_2}\int_{\Omega_1}(u\eta_t+u^m\Delta\eta)\,dxdt+\int_{t_1}^{t_2}\int_{\Sigma_1}f\eta\,d\sigma dt\notag\\
&\qquad\qquad\qquad\qquad =\int_{t_1}^{t_2}\int_{\Sigma_2}g^m\frac{\1\eta}{\1\nu}\,d\sigma dt+\int_{\Omega_1}u(x,t_2)\eta (x,t_2)\,dx-\int_{\Omega_1}u(x,t_1)\eta (x,t_1)\,dx
\end{align}
for any $0<t_1<t_2<T$, and $\eta\in C^2(\2{\Omega}_1\times (0,T))$ satisfying $\eta =0$ on
$\Sigma_2\times (0,T)$, and $\1\eta/\1\nu =0$ on $\Sigma_1\times (0,T)$ and $u$ has initial value $u_0$. We say that $u$ is a solution
(subsolution, supersolution respectively) of \eqref{Dirichlet-Neumann-problem} if $u\in L^{\infty}_{loc}(\overline{\Omega}_1\times(0,T)) $ is positive in $\Omega_1\times (0,T)$ and satisfies \eqref{fast-diffusion-eqn} in $\Omega_1\times (0,T)$ ($\le$, $\ge$ respectively) in the classical sense with 
\begin{equation}\label{u-L1-converge-u0-as-t-goto-0}
\|u(\cdot,t)-u_0\|_{L^1(\Omega_{\delta})}\to 0\quad\mbox { as }t\to 0
\end{equation}
and also satisfies \eqref{weak-soln-defn} ($\ge$, $\le$ respectively) for any $0<t_1<t_2<T$, and $\eta\in C^2(\2{\Omega}_1\times (0,T))$ satisfying $\eta =0$ on $\Sigma_2\times (0,T)$, and $\1\eta/\1\nu =0$ on $\Sigma_1\times (0,T)$. 

We say that $u$ is a solution
(subsolution, supersolution respectively) of \eqref{singular-Neumann-problem} if $u\in L^{\infty}_{loc}((\overline{\Omega}\setminus\{a_1,\dots,a_{i_0}\})\times(0,T))\cap C^{2,1}((\overline{\Omega}\setminus\{a_1,\dots,a_{i_0}\})\times(0,T))$ is positive in $\hat{\Omega}\times (0,T)$ and satisfies \eqref{fast-diffusion-eqn} in $\hat{\Omega}\times (0,T)$ ($\le$, $\ge$ respectively) in the classical sense with 
\begin{equation}\label{u-L1-converge-u0-as-t-goto-0-Omega-hat}
\lim_{t\to 0}\int_{\hat{\Omega}}u(x,t)\eta (x)\,dx=\int_{\hat{\Omega}}u_0\eta\,dx\quad\forall\eta\in C_0^{\infty}(\hat{\Omega})
\end{equation}
and also satisfies 
\begin{equation}\label{weak-singular-soln-defn}
\int_{t_1}^{t_2}\int_{\hat{\Omega}}(u\eta_t+u^m\Delta\eta)\,dxdt+\int_{t_1}^{t_2}\int_{\1\Omega}f\eta\,d\sigma dt
=\int_{\Omega_1}u(x,t_2)\eta (x,t_2)\,dx-\int_{\Omega_1}u(x,t_1)\eta (x,t_1)\,dx
\end{equation}
($\ge$, $\le$ respectively) for any $0<t_1<t_2<T$, and $\eta\in C_0^2((\overline{\Omega}\setminus\{a_1,\dots,a_{i_0}\})\times(0,T))$ satisfying $\1\eta/\1\nu =0$ on $\partial\Omega\times (0,T)$. 

We say that $u$ is a solution
(subsolution, supersolution respectively) of \eqref{singular-Neumann-problem2} if $u\in L^{\infty}_{loc}(\hat{\R}\times(0,T))\cap C^{2,1}(\hat{\R}\times(0,T))$ is positive in $\hat{\R}\times (0,T)$ and satisfies \eqref{fast-diffusion-eqn} in $\hat{\R}\times (0,T)$ ($\le$, $\ge$ respectively) in the classical sense with 
\begin{equation}\label{u-L1-converge-u0-as-t-goto-0-Omega-hat*}
\lim_{t\to 0}\int_{\hat{\Omega}}u(x,t)\eta (x)\,dx=\int_{\hat{\Omega}}u_0\eta\,dx\quad\forall\eta\in C_0^{\infty}(\hat{\R}).
\end{equation}

For any $x_0\in\R^n$,
$x_0'\in\R^{n-1}$, $R>0$, we let $B_R(x_0)=\{x\in\R^n:|x-x_0|<R\}$, $B'_R(x_0')=\{x'\in\R^{n-1}:|x-x_0'|<R\}$, $\hat{B}_R(x_0)=B_R(x_0)\setminus\{x_0\}$, $B_R=B_R(0)$, $\hat{B}_R=\hat{B}_R(0)$ and $B_R'=B_R'(0)$. For any $a\in\R$, we let $a_+=\max (a,0)$ and $a_-=\max (-a,0)$. For any set $A\in\R^n$, we let $\chi_A$ be the characteristic function of the set $A$.

\section{A priori estimates}
\setcounter{equation}{0}
\setcounter{thm}{0}

In this section we will prove some a priori estimates for the solutions of \eqref{Neumann-holes-eqn}. We will also prove a $L^p-L^{\infty}$ estimates for the solutions of \eqref{Neumann-holes-eqn} for some constant $p>1$. These $L^p-L^{\infty}$ estimates will be used in section three to give uniform upper bound for the approximating $C^{2,1}$ solutions of \eqref{Neumann-holes-eqn} which appear in the construction of solution of \eqref{Neumann-holes-eqn}.
Note that 
similar $L^{\infty}$ estimates are also obtained in \cite{BV2}, \cite{CD}, \cite{CV}, \cite{D}, \cite{DK},\cite{DGV1}, \cite{DGV2}, \cite{DK}, \cite{DKV} and \cite{HP}. 
We first observe that by an argument 
similar to the proof of Lemma 2.3 of \cite{DaK} we have the following result.

\begin{lemma}[cf. Lemma 1.1 of \cite{Hs2}]\label{comparsion-lemma1}
Let $\Omega_1\subset\R^n$ be a smooth bounded domain and let $\Sigma_1$, $\Sigma_2$ be relatively open subsets of $\partial\Omega_1$ such that $\partial\Omega_1=\Sigma_1\cup\Sigma_2$ and if $n\ge 2$ then $\overline{\Sigma}_1\cap\overline{\Sigma}_2$ is a $C^2$ manifold of dimension $n-2$. Let $0\leq u_{0,1}$, $u_{0,2}\in L^{1}(\Omega)$, $f_1$, $f_2\in L^1(\Sigma_1\times(0,T))$ and $g_1$, $g_2\in C(\Sigma_2\times(0,T))$ be such that $0\le g_1\leq g_2$ on $\Sigma_2\times(0,T)$. Suppose $u_1$, $u_2$, are subsolution and supersolution of
\eqref{Dirichlet-Neumann-problem} in $\Omega_1\times(0,T)$ with $f=f_1$, $f_2$, $g=g_1$, $g_2$ and
$u_0=u_{0,1}$, $u_{0,2}$, respectively. Then
\begin{equation*}
\int_{\Omega_1}(u_1-u_2)_+(x,t)\,dx\leq \int_{\Omega_1}(u_{0,1}-u_{0,2})_+(x,t)\,dx+\int_0^t\int_{\Sigma_1}(f_1-f_2)_+\,d\sigma ds
\qquad \forall 0\leq t<T.
\end{equation*}
\end{lemma}

By the same argument as the proof of Lemma 3.4 of \cite{Hu1} we have the following result.

\begin{lemma}[cf. Lemma 3.4 of \cite{Hu1}]\label{positive-angle-at-bdary-lemma}
Suppose $\Omega\subset\R^n$ is a smooth bounded convex domain. For any $x\in\1\Omega$, $x_0\in\Omega$, let $n(x)$ be the unit outward normal vector at $x$ with respect to $\Omega$ and let $n_1(x)$ be the unit vector along the line segment $\overrightarrow{x_0x}$ from $x_0$ to $x$. If $\theta (x)$ is the
angle between $n(x)$ and $n_1(x)$, then there exists a constant
$0<c_0\le 1$ such that
\begin{equation*}
0<c_0\le\text{cos }\theta (x)\le 1\quad\forall x\in\1\Omega.
\end{equation*}
\end{lemma}
Now, we are going to prove some estimates for the solutions of \eqref{Neumann-holes-eqn}.

\begin{lemma}\label{soln-bded-holes-lemma}
Let $n\ge 1$, $0<m<1$, $T>0$, $q\ge\frac{2}{1-m}$, $0<\delta_1<\min\left(\frac{(1-m)q}{4+(1-m)q},\delta_0\right)$, and 
$0<\delta\le\delta_2<\left(\frac{(1-m)q}{4+(1-m)q}\right)\delta_1$. Let $0\leq u_0\in 
L^1(\overline{\Omega}_{\delta})$ such that 
\begin{equation}\label{u0-blow-up-condition}
u_0(x)\le\frac{C_2}{|x-a_i|^{q}} \qquad \forall \delta\le |x-a_i|<\delta_1, i=1,\cdots,i_0
\end{equation}
holds for some constant $C_2>0$. Let $f\in L^1(\partial\Omega\times(0,T))$ and $g_i\in L^1(\partial B_{\delta}(a_i)\times(0,T))$,
$i=1,2,\dots,i_0$, be such that $\sup f<\infty$, $\sup g_i<\infty$, for all $i=1,2,\dots,i_0$.
Let $u$ be a solution of 
\begin{equation}\label{Neumann-problem-holes}
\left\{\begin{aligned}
u_t=&\La u^m\qquad \mbox{ in }\Omega_{\delta}\times(0,T)\\
\frac{\partial u^m}{\partial\nu}=&f\qquad\quad\, \mbox{ on }\partial\Omega\times(0,T)\\
\frac{\partial u^m}{\partial\nu}=&\frac{g_i}{\delta^{mq+1}}\quad\mbox{ on }B_{\delta}(a_i)\times(0,T)\qquad\forall i=1,\dots,i_0\\
u(x,0)=&u_0(x)\qquad\mbox{in }\Omega_{\delta}.
\end{aligned}\right.
\end{equation}
Then there exists a constant $A_1>0$ such that
\begin{equation}\label{u-upper-bd-phi-tidle}
u(x,t)\leq \phi_{A_1}(x-a_i,t) \qquad \forall\delta\le|x-a_i|<\delta_1,\,\,0\leq t<T, i=1,\cdots,i_0
\end{equation}
holds for all $0<\delta\leq \delta_2$ where
\begin{equation}\label{phi-A-defn}
\phi_{A_1}(x,t)=\frac{A_1(1+t)^{\frac{1}{1-m}}}{|x|^{q}(\delta_1-|x|)^{\frac{2}{1-m}}}.
\end{equation}
If $\Omega$ is a smooth convex domain and 
\begin{equation}\label{u0-L-infty-bd-condition}
\|u_0\|_{L^{\infty}(\overline{\Omega}_{\delta})}\le M_0
\end{equation} 
holds for some constant $M_0>0$, then there exists a constant $M_1>0$ 
depending on $M_0$ such that
\begin{equation}\label{u-upper-bd1}
u(x,t)\leq M_1 \quad \forall (x,t)\in\overline{\Omega}_{\delta_2}\times[0,T)
\end{equation}
holds for any $0<\delta\leq\delta_2$.
\end{lemma}
\begin{proof}
We will use a modification of the proof of Lemma 1.3 of \cite{Hs2} to prove the lemma. Without loss of generality it suffices to prove \eqref{u-upper-bd-phi-tidle} for $i=i_0=1$. Let
\begin{equation}\label{eq-condition-of-A-1-by-max-1}
A_1=\max\left\{C_2,\left(\frac{m(mq^2+2q+2n+4)}{1-m}\right)^{\frac{1}{1-m}},\,\,\left(\frac{2(\sup g_1)_+}{mq}\right)^{\frac{1}{m}}\right\}
\end{equation}
and let $\delta'\in(\delta_2,\delta_1)$. For any $0<t_1<t_2<T$, let
\begin{equation}\label{M-defn}
M=\max_{\substack{\delta'\leq |x-a_1|\leq\delta_1\\t_1\leq t\leq t_2}}\frac{m|\nabla u|}{u^{1-m}}.
\end{equation}
Since
\begin{equation*}
\frac{\partial \phi_{A_1}^{m}}{\partial r}(x-a_1,t)=\frac{mA_1^m(1+t)^{\frac{m}{1-m}}}{r^{mq}(\delta_1-r)^{\frac{2m}{1-m}}}\left(-\frac{q}{r}+\frac{2}{(1-m)(\delta_1-r)}\right)\to\infty
\end{equation*}
uniformly on $t\in [0,\infty]$ as $r=|x-a_1|\to\delta_1^-$,
there exists a constant $\delta''\in(\delta',\delta_1)$ such that
\begin{equation}\label{eq-condition-of-tilde-phi-on-boundary-1}
\frac{\partial \phi_{A_1}^{m}}{\partial r}(x-a_1,t)>M\quad\forall\delta''\leq |x-a_1|\leq \delta_1,t\geq 0.
\end{equation}
Since $q/\delta\ge 4/[(1-m)(\delta_1-\delta)]$, by direct computation,
\begin{align}\label{eq-condition-of-tilde-phi-on-boundary-2}
\frac{\partial \phi_{A_1}^{m}}{\partial \nu}(x-a_1,t)&=\frac{mA_1^m(1+t)^{\frac{m}{1-m}}}{{\delta}^{mq}(\delta_1-\delta)^{\frac{2m}{1-m}}}\left[\frac{q}{\delta}-\frac{2}{(1-m)(\delta_1-\delta)}\right]\notag\\
&\geq \frac{mA_1^mq}{2{\delta}^{mq+1}(\delta_1-\delta)^{\frac{2m}{1-m}}}\notag\\
&\geq \frac{mA_1^mq}{2{\delta}^{mq+1}}\geq \frac{g_1}{{\delta}^{mq+1}} \qquad \quad \mbox{on }\partial B_{\delta}(a_1)\times(0,\infty).
\end{align}
By \eqref{eq-condition-of-A-1-by-max-1},
\begin{equation}\label{eq-condition-of-tilde-phi-in-interior-1}
\begin{aligned}
\La\phi^m_{A_1}
=&A_1^m(1+t)^{\frac{m}{1-m}}\left[\frac{mq(mq-n+2)}{|x|^{mq+2}(\delta_1-|x|)^{\frac{2m}{1-m}}}-\frac{2m(2mq-n+1)}{(1-m)|x|^{mq+1}(\delta_1-|x|)^{\frac{1+m}{1-m}}}+\frac{2m(1+m)}{(1-m)^2|x|^{mq}(\delta_1-|x|)^{\frac{2}{1-m}}}\right]\\
\le&\frac{mA_1^m(mq^2+2q+2n+4)(1+t)^{\frac{m}{1-m}}}{(1-m)^2|x|^{mq+2}(\delta_1-|x|)^{\frac{2}{1-m}}}\\
\le&\phi_{A_1,t}\qquad\qquad\qquad\qquad\qquad\qquad \forall 0<|x|<\delta_1, t\geq 0.
\end{aligned}
\end{equation}
Hence by \eqref{eq-condition-of-tilde-phi-on-boundary-1}, \eqref{eq-condition-of-tilde-phi-on-boundary-2} and \eqref{eq-condition-of-tilde-phi-in-interior-1} for any $\delta_1'\in (\delta'',\delta_1)$, $\phi_{A_1}(x-a_1,t)$ is a supersolution of
\begin{equation}\label{w-eqn}
\begin{cases}
w_t=\La w^m\qquad\quad \mbox{ in }(B_{\delta_1'}(a_1)\bs B_{\delta}(a_1))\times(0,T)\\
\frac{\partial w^m}{\partial\nu}=M \qquad \quad\,\, \mbox{ on }\partial B_{\delta_1'}(a_1)\times(0,T)\\
\frac{\partial w^m}{\partial\nu}=\frac{g_1}{\delta^{mq+1}}\qquad\,\,\mbox{ on }\partial B_{\delta}(a_1)\times(0,T)\\
w(x,0)=u_0(x)\quad\mbox{ in }B_{\delta_1'}(a_1)\bs B_{\delta}(a_1)
\end{cases}
\end{equation}
where $\1/\1\nu$ is the derivative with respect to the unit outward normal at the boundary of the domain $B_{\delta_1'}(a_1)\bs B_{\delta}(a_1)$.
Since $u$ is a subsolution of \eqref{w-eqn}, by Lemma \ref{comparsion-lemma1}, $\forall 0<t_1<t_2<T$, $\delta_1'\in (\delta'',\delta_1)$,
\begin{align*}
&\int_{B_{\delta_1'}(a_1)\bs B_{\delta}(a_1)}(u(x,t)-\phi_{A_1}(x-a_1,t))_+\,dx\leq \int_{B_{\delta_1'}(a_1)\bs B_{\delta}(a_1)}(u(x,t_1)-\phi_{A_1}(x-a_1,t_1))_+\,dx\quad\forall t_1\le t\le t_2\\
\Rightarrow\quad&\int_{B_{\delta_1}(a_1)\bs B_{\delta}(a_1)}(u(x,t)-\phi_{A_1}(x-a_1,t))_+\,dx\leq \int_{B_{\delta_1}(a_1)\bs B_{\delta}(a_1)}(u(x,t_1)-\phi_{A_1}(x-a_1,t_1))_+\,dx\quad\mbox{ as }
\delta_1'\to\delta_1\\
\Rightarrow\quad&\int_{B_{\delta_1}(a_1)\bs B_{\delta}(a_1)}(u(x,t)-\phi_{A_1}(x-a_1,t))_+\,dx\le 0\qquad\forall 0<t<T\qquad\qquad\qquad\mbox{ as }t_1\to 0, 
t_2\to T
\end{align*}
and \eqref{u-upper-bd-phi-tidle} follows.\\
\indent Suppose now $\Omega$ is a smooth convex domain and \eqref{u0-L-infty-bd-condition} holds for some constant $M_0>0$. Let $n(x)$, $n_1(x)$, $\theta(x)$ and $c_0$ be as in Lemma \ref{positive-angle-at-bdary-lemma} with $x_0=a_1$. By \eqref{u-upper-bd-phi-tidle} there exists a constant $M>0$ such that
\begin{equation*}
u(x,t)\le M\qquad \forall (x,t)\in\left(\bigcup_{i=1}^{i_0} \partial B_{\delta_2}(a_i)\right)\times(0,T)
\end{equation*}
holds for all $0<\delta\le\delta_2$. Let
\begin{equation*}
w(x,t)=A_2(1+t)^{\frac{1}{1-m}}e^{\frac{|x-a_1|}{\delta_2}}
\end{equation*}
where
\begin{equation*}
A_2=\max\left\{\left(\frac{m(1-m)(m+n-1)}{\delta_2^2}\right)^{\frac{1}{1-m}},\left(\frac{\delta_2}{mc_0}(sup f)_+\right)^{\frac{1}{m}},\,\,M,\,\,\|u_0\|_{L^{\infty}(\Omega_{\delta_2})}\right\}.
\end{equation*}
Then
\begin{equation*}
w(x,0)\ge u_0(x)\quad\mbox{ in }\Omega_{\delta_2}
\end{equation*}
and 
\begin{equation*}
w(x,t)\ge u(x,t)\quad\mbox{on }\left(\bigcup_{i=1}^{i_0}\partial B_{\delta_2}(a_i)\right)\times(0,T).
\end{equation*}
By direct computation,
\begin{equation*}
\begin{aligned}
\La w^m=&\frac{mA_2^m}{\delta_2}\left[\frac{m}{\delta_2}+\frac{n-1}{|x-a_1|}\right](1+t)^{\frac{m}{1-m}}e^{\frac{m|x-a_1|}{\delta_2}}
\le\frac{mA_2^m(m+n-1)}{\delta_2^2}(1+t)^{\frac{m}{1-m}}e^{\frac{m|x-a_1|}{\delta_2}}\\
\le&\frac{A_2}{1-m}(1+t)^{\frac{m}{1-m}}e^{\frac{|x-a_1|}{\delta_2}}=w_t, \qquad \mbox{in }\Omega_{\delta_2}\times(0,T).
\end{aligned}
\end{equation*}
Moreover by Lemma \ref{positive-angle-at-bdary-lemma},
\begin{equation*}
\frac{\partial w^m}{\partial\nu}(x,t)=\frac{\partial w^m}{\partial n_1}(x,t)\cdot\cos\theta(x)
=\frac{mA_2^{m}}{\delta_2}\cos\theta(x)(1+t)^{\frac{m}{1-m}}e^{\frac{m|x-a_1|}{\delta_2}}\ge\frac{mA_2^mc_0}{\delta_2}
\ge f(x,t)
\end{equation*}
for all $x\in\partial\Omega$, $0<t<T$.
Hence by Lemma \ref{comparsion-lemma1},
\begin{equation*}
u(x,t)\leq w(x,t)\leq A_2(1+T)e^{M_2}\quad\mbox{ in }\overline{\Omega}_{\delta_2}\times[0,T)
\end{equation*}
where $M_2=\max_{x\in\overline{\Omega}_{\delta_2}}\left(\frac{|x-a_1|}{\delta_2}\right)$ and \eqref{u-upper-bd1} follows.
\end{proof}

\begin{lemma}\label{comparison-lem2}
Let $n\ge 1$, $0<m<1$, $0<\delta<\delta_1<\min (\delta_0, m/2)$, $0\le f\in L_{loc}^{\infty}(\1\Omega\times [0,\infty))$ and $g_i\in L_{loc}^{\infty}([0,\infty))$, $i=1,\cdots, i_0$, be such that 
\begin{equation}\label{gi-lower-bd-condition}
\inf_{[0,\infty)} g_i(t)>m(q+4\delta_1^{-2})\quad\forall i=1,\cdots, i_0.
\end{equation} 
Let $0\le u_0\in L^1(\Omega_{\delta})$ be such that 
\begin{equation}\label{u0-lower-blow-up-condition}
u_0(x)\ge\frac{C_1}{|x-a_i|^qe^{\frac{1}{\delta_1^2-|x-a_i|^2}}}\quad \forall \delta\le |x-a_i|<\delta_1, i=1,\cdots,i_0
\end{equation} 
for some constants $C_1>0$  and $q\ge\max\left(\frac{n}{2m},\frac{n-2}{m}\right)$ and let $u$ be a solution of \eqref{Neumann-problem-holes}. Then 
\begin{equation}\label{singular-soln-lower-bd}
u(x,t)\ge\frac{C_1}{|x-a_i|^qe^{\frac{1}{\delta_1^2-|x-a_i|^2}}}\quad\forall \delta\le |x|<\delta_1,t>0, i=1,\cdots,i_0
\end{equation}
holds for any $0<\delta\le\delta_1/2$.
\end{lemma}
\begin{proof}
Without loss of generality it suffices to prove the lemma when $i_0=1$ and $a_1=0$. Let
\begin{equation*}
\phi(x)=|x|^{-q}\psi(x), \qquad\psi(x)=e^{-\frac{1}{\delta_1^2-|x|^2}}.
\end{equation*}
By direct computation,
\begin{equation*}
\Delta |x|^{-mq}=mq(mq+2-n)|x|^{-mq-2}\ge 0,
\end{equation*}
\begin{equation*}
\Delta \psi^m=\frac{2m\psi^m}{(\delta_1^2-|x|^2)^4}\left(2m|x|^2-4|x|^2(\delta_1^2-|x|^2)-n(\delta_1^2-|x|^2)^2\right)
\end{equation*}
Hence
\begin{align}\label{phi^m-subharmonic}
\Delta\phi^m=&\psi^m\Delta |x|^{-mq}+2\nabla\psi^m\cdot\nabla |x|^{-mq}+|x|^{-mq}\Delta\psi^m\notag\\
\ge&\frac{4m^2q|x|^{-mq}\psi^m}{(\delta_1^2-|x|^2)^2}+|x|^{-mq}\Delta\psi^m\notag\\
\ge&\frac{2m|x|^{-mq}\psi^m}{(\delta_1^2-|x|^2)^4}\left((2m-4\delta_1^2)|x|^2+(2mq-n)(\delta_1^2-|x|^2)^2\right)\notag\\
\ge&0\quad\mbox{ in }B_{\delta_1}.
\end{align}
By \eqref{gi-lower-bd-condition},
\begin{equation}\label{phi-derivative-near-zero}
\left.\frac{\1\phi^m}{\1\nu}\right|_{|x|=\delta}=\frac{m\psi(\delta)^m}{\delta^{qm+1}}\left(q+\frac{2}{(\delta_1^2-\delta^2)^2}\right)
\le\frac{g_i}{\delta^{qm+1}}\quad\forall0<\delta<\frac{\delta_1}{2}.
\end{equation}
By \eqref{phi^m-subharmonic} and \eqref{gi-lower-bd-condition} for any $0<\delta<\frac{\delta_1}{2}$, $\phi$ is a subsolution of 
\begin{equation}\label{dirichlet-neumann-problem}
\left\{\begin{aligned}
u_t=&\La u^m\qquad \mbox{ in }B_{\delta_1}\times(0,\infty)\\
\frac{\partial u^m}{\partial\nu}=&\frac{g_i}{\delta^{mq+1}}\quad\mbox{ on }B_{\delta}\times(0,\infty)\\
u=&0\qquad\quad\, \mbox{ on }\partial B_{\delta_1}\times(0,\infty)\\
u(x,0)=&u_0(x)\qquad\mbox{in }B_{\delta_1}\setminus B_{\delta}.
\end{aligned}\right.
\end{equation}
Since $u$ is a supersolution of \eqref{dirichlet-neumann-problem}, by Lemma \ref{comparsion-lemma1} \eqref{singular-soln-lower-bd} follows.
\end{proof}

We will now prove a $L^p-L^{\infty}$ estimates for the solution of \eqref{Neumann-problem-holes}.
Since the proof is similar to the proof in section 1 of \cite{Hu1}, we will only sketch the argument here. For any smooth bounded domain 
$\Omega\subset\R^n$, let $x_0\in\partial\Omega$. When $n\ge 2$, by rotation and translation of the coordinate axis we may assume that $x_0$ is at the origin and the tangent plane to $\partial\Omega$ at $x_0$ is $\R^{n-1}\times\{0\}$ and there exists a constant $0<R_0\le\delta_0/2$ and a smooth function
\begin{equation*}
\phi_1:\R^{n-1}\to \R
\end{equation*}
with $\phi_1\in C_0^{\infty}(\R^{n-1})$, $\phi_1(0)=0$, $\nabla\phi_1(0)=0$, and
$\sup_{x'\in \R^{n-1}}|\nabla\phi_1(x')|\le\frac{1}{10}$ such that
\begin{equation*}
\begin{cases}
\begin{aligned}
\Omega\cap B_{R_0}=&\left\{(x',x_n):x'\in\R^{N-1}, x_n>\phi_1(x')\right\}\cap B_{R_0}\\
\partial\Omega\cap B_{R_0}=&\left\{(x',x_n):x'\in\R^{N-1}, x_n=\phi_1(x')\right\}\cap B_{R_0}.
\end{aligned}
\end{cases}
\end{equation*}
Then $(\cup_{i=1}^{i_0}B_{\delta_0/2}(a_i))\cap B_{R_0}(x_0)=\phi$. For any $x=(x',x_n)\in \Omega\cap B_{R_0}$ 
with $x'\in\R^{n-1}$, $x_n\in\R$, let
\begin{equation}\label{domain-transformation}
\psi_1(x)=(x',x_n-\phi_1(x'))
\end{equation}
and $\Omega'_1=\psi_1(\Omega\cap B_{R_0})$. Then $\psi_1$ is a differeomorphism between $\Omega\cap B_{R_0}$ and
$\Omega'_1$. Let
\begin{equation*}
\tilde{u}(y,t)=u(y',y_n+\phi_1(y'),t) \qquad \forall (y',y_n)\in\Omega'_1.
\end{equation*}
For $0<r<R_0/5$, let
\begin{equation*}
D_r'=\left\{\begin{aligned}
&\{(y',y_n)\in\R^{n-1}\times\R:|y'|<r, 0<y_n<r\}\quad\mbox{ if }n\ge 2\\
&\{y\in\R:|y-x_0|<r\}\qquad\qquad\qquad\qquad\qquad\mbox{if }n=1
\end{aligned}\right.
\end{equation*}
and $D_r=\psi_1^{-1}(D_r')$  where $\psi_1$ is given by \eqref{domain-transformation} if $n\ge 2$ and $\psi_1$ is the identity map if $n=1$.
For $0<r<\rho<R_0/5$ and $0<t_1<t_2<t_0<T$, $t_2<t<T$, we let $S(t)=\psi_1^{-1}(D_r')\times(t_2,t]$, $R(t)=\psi_1^{-1}(D_{\rho}')\times(t_1,t]$,  $S=S(t_0)$ and $R=R(t_0)$.

\begin{lemma}\label{weak-w12-soln-lem}
Let $0<\delta<\delta_0$, $f\in L^1(\partial\Omega\times(0,T))$, and $g_i\in L^1(\partial B_{\delta}(a_i)\times(0,T))$,
$i=1,2,\dots,i_0$. Suppose $u\in C^{2,1}(\2{\Omega}_{\delta}\times (0,T))$ is a solution of \eqref{Neumann-holes-eqn} in $\Omega_{\delta}\times (0,T)$. For any $a>0$, let $v=\max (u,a)$. Then
\begin{align}\label{weak-w12-soln}
&\int_{\Omega_{\delta}}v(x,t_2)\phi (x,t_2)\,dx+\int_{t_1}^{t_2}\int_{\Omega}\nabla v^m\cdot\nabla\phi \,dx\,dt\notag\\
\le&\int_{\Omega_{\delta}}v(x,t_1)\phi (x,t_1)\,dx+\int_{t_1}^{t_2}\int_{\Omega}v\phi_t\,dx\,dt+\int_{t_1}^{t_2}\int_{\1\Omega}f\phi\,d\sigma\,dt+\int_{t_1}^{t_2}\int_{\cup_{i=0}^{i_0}\1 B_{\delta}(a_i)}g_i\phi\,d\sigma\,dt\notag\\
\end{align}
holds for any $0<t_1<t_2<T$ and $\phi\in C(\2{\Omega}_{\delta}\times (0,T))$ such that the weak derivatives $\phi_t,\nabla\phi$, exist and belong to $L^2(\Omega_{\delta}\times (0,T))$ and
\begin{equation}\label{weak-w12-soln2}
\int_{\Omega_{\delta}}v_t\phi\,dx+\int_{\Omega}\nabla v^m\cdot\nabla \phi\,dx
\le\int_{\1\Omega}f\phi\,d\sigma+\int_{\cup_{i=0}^{i_0}\1 B_{\delta}(a_i)}g_i\phi\,d\sigma
\end{equation}
holds for any $0<t<T$ and $\phi\in C(\2{\Omega}_{\delta}\times (0,T))$ such that for any $0<t'<T$ the weak derivatives $\nabla\phi (x,t')$ exists and belong to $L^2(\Omega_{\delta})$.
\end{lemma}
\begin{proof}
By approximation it suffices to show that \eqref{weak-w12-soln} holds for any $\phi\in C^{\infty}(\2{\Omega}_{\delta}\times (0,T))$. Since the proof of the lemma is similar to the proof of Lemma 1.1 of \cite{Hu1} we will only sketch the argument here.
We choose a sequence of functions $\{p_k\}_{k=1}^{\infty}\subset C^{\infty}(\R)$, $0\le p_k\le 1$, $p_k(s)=0$ for 
all $s\le a+(2k)^{-1}$, $p_k(s)=1$ for all $s\ge a+k^{-1}$, $0\le p_k'(s)\le Ck$ for some constant $C>0$, and $p_k(s)\to\chi_{(a,\infty)}(s)$ as $k\to\infty$. Multiplying 
\eqref{fast-diffusion-eqn} by $p_k(u)\phi$ and integrating by parts,
\begin{align}\label{weak-soln-k-approx-eqn}
&\int_{\Omega}q_k(u(x,t_2))\phi(x,t_2)\,dx+\int_{t_1}^{t_2}\int_{\Omega}\nabla u^m\cdot\nabla \phi\,dx\,dt\notag\\
\le&\int_{\Omega}q_k(u(x,t_2))\phi(x,t_2)\,dx+\int_{t_1}^{t_2}\int_{\Omega}\nabla u^m\cdot\nabla (p_k(u)\phi)\,dx\,dt\notag\\
\le&\int_{\Omega}q_k(u(x,t_1))\phi(x,t_1)\,dx+\int_{t_1}^{t_2}\int_{\Omega}q_k(u)\phi_t\,dx\,dt+\int_0^T\int_{\1\Omega}f\phi\,d\sigma\,dt+\int_{t_1}^{t_2}\int_{\cup_{i=1}^{i_0}\1 B_{\delta}(a_i)}g_i\phi\,d\sigma\,dt
\end{align}
where
$$
q_k(s)=\int_a^sp_k(\hat{s})\,d\hat{s}
$$
Since $q_k(u)\to v-a$ as $k\to\infty$, letting $k\to\infty$ in \eqref{weak-soln-k-approx-eqn} we get \eqref{weak-w12-soln}. By a similar argument we get \eqref{weak-w12-soln2} and the lemma follows.
\end{proof}

\begin{lemma}\label{lem-the-first-lemma-for-moser-iteration-i3}
Let $0<\delta<\delta_0$, $f\in L^1(\partial\Omega\times(0,T))$, and $g_i\in L^1(\partial B_{\delta}(a_i)\times(0,T))$,
$i=1,2,\dots,i_0$, be such that $f\leq M$ for some constant $M>0$. Suppose $u\in C^{2,1}(\2{\Omega}_{\delta}\times (0,T))$ is a solution of \eqref{Neumann-holes-eqn} in $\Omega_{\delta}\times (0,T)$. Let $v=\max (u,1)$. Then there exists a constant $C>0$ such that
\begin{equation}\label{step1-ineqn}
\sup_{t_2\le t\le t_0}\int_{D_r}v^{\alpha+1}(x,t)\,dx+\iint_{S}|\nabla v^{\frac{\alpha+m}{2}}|^2\,dx\,dt\leq C\{(\rho-r)^{-2}+(t_2-t_1)^{-1}\}(1+\alpha^2)\iint_{R}v^{\alpha+1}\,dx\,dt
\end{equation}
holds for any $\alpha>0$, $0<\rho<R_0/5$, and $0<m<1$.
\end{lemma}
\begin{proof}
Since the proof is similar to that of Lemma 1.1 of \cite{Hu1}, we will only sketch the argument here.
We first choose $\eta\in C^{\infty}(\R^{n+1})$, $0\le\eta\le 1$, such that $\eta(x,s)=0$ for all $(x,s)\not\in R$, $\eta(x,s)=1$ for all $(x,s)\in S$, and $|\eta_t|\le C(t_2-t_1)^{-1}$, $|\nabla\eta|\le C(\rho-r)^{-1}$,
for some constant $C>0$. By Lemma \ref{weak-w12-soln-lem} the function $v=\max(u,1)$ satisfies \eqref{weak-w12-soln2}.
Since the function $v^{\alpha}\eta^2\in C(\2{\Omega}_{\delta}\times (0,T))$ and its first order weak derivatives belong to $L^2({\Omega}_{\delta}\times (0,T))$, putting $\phi=v^{\alpha}\eta^2$ in \eqref{weak-w12-soln2} and simplifying
\begin{equation}\label{eq-aligned-the-first-step-for-energy-inequality-34t}
\begin{aligned}
&\frac{1}{\alpha+1}\int_{D_r}v^{\alpha+1}(x,t)\,dx+\iint_{R(t)}\nabla (v^{m})\cdot
\nabla(v^{\alpha}\eta^2)\,dxds\\
\le&\frac{2}{\alpha+1}\iint_{R(t)}v^{\alpha+1}\eta\eta_t\,dxds+M\iint_{(\partial D_{\rho}\cap\partial\Omega)\times(t_1,t)}v^{\alpha}\eta^2\,d\sigma ds.
\end{aligned}
\end{equation}
Then
\begin{equation*}
\begin{aligned}
&\frac{1}{\alpha+1}\int_{D_r}v^{\alpha+1}(x,t)\,dx
+\frac{4\alpha m}{(\alpha+m)^2}\iint_{R(t)}\eta^2|\nabla v^{\frac{\alpha+m}{2}}|^2\,dxds\\
\le &\frac{2}{\alpha+1}\iint_{R(t)}v^{\alpha+1}\eta\eta_t\,dxds+2m\iint_{R(t)}v^{\alpha+m-1}\eta\left|\nabla v\right|\left|\nabla\eta\right|\,dxds
+M\iint_{(\partial D_{\rho}\cap\partial\Omega)\times(t_1,t_0)}v^{\alpha}\eta^2\,d\sigma ds\\
\le&\frac{2}{\alpha+1}\iint_{R(t)}v^{\alpha+1}\eta\eta_t\,dxds+\frac{2m\alpha}{(\alpha+m)^2}
\iint_{R(t)}\eta^2|\nabla v^{\frac{\alpha+m}{2}}|^2\,dxds
+\frac{2m}{\alpha}\iint_{R(t)}v^{\alpha+m}|\nabla\eta|^2\,dxds\\
&\qquad  +M\iint_{(\partial D_{\rho}\cap\partial\Omega)\times(t_1,t)}v^{\alpha}\eta^2\,d\sigma ds.
\end{aligned}
\end{equation*}
Hence
\begin{equation}\label{ineqn-1}
\frac{\alpha}{\alpha+1}\int_{D_r}v^{\alpha+1}(x,t)\,dx
+\frac{2m\alpha^2}{(\alpha+m)^2}\iint_{R(t)}\eta^2|\nabla v^{\frac{\alpha+m}{2}}|^2\,dxds\\
\le 2\iint_{R(t)}v^{\alpha+1}(|\eta_t|+|\nabla\eta|^2)\,dxds
+ MI
\end{equation}
where
\begin{equation*}
I=\alpha\iint_{(\partial D_{\rho}\cap\partial\Omega)\times(t_1,t)}v^{\alpha}\eta^2
\,d\sigma ds.
\end{equation*}
By an argument similar to the proof of Lemma 1.1 of \cite{Hu1},
\begin{equation}\label{bd-estimate}
\begin{aligned}
I&\le\alpha\int_{t_1}^t\int_{\Omega}|\partial_{x_n}(v^{\alpha}\eta^2)|\,dxdt\\
&\le\alpha^2\iint_{R(t)}v^{\alpha-1}\eta^2|\nabla v|\,dxds
+2\alpha\iint_{R(t)}v^{\alpha}\eta|\nabla\eta|\,dxds\\
&\le\alpha^2\iint_{R(t)}v^{\alpha+m-1}\eta^2|\nabla v|\,dxds
+2\alpha\iint_{R(t)}v^{\alpha}\eta|\nabla\eta|\,dxds\\
&\le\frac{2\alpha^2}{\alpha+m}\iint_{R(t)}\eta^2v^{\frac{\alpha+m}{2}}|\nabla v^{\frac{\alpha+m}{2}}|\,dxds
+2\alpha\iint_{R(t)}v^{\alpha}\eta|\nabla\eta|\,dxds\\
&\le\frac{m\alpha^2}{(\alpha+m)^2(M+1)}\iint_{R}\eta^2\left|\nabla v^{\frac{\alpha+m}{2}}\right|^2\,dxds
+\alpha^2m^{-1}(M+1)\iint_{R} v^{\alpha+m}\eta^2\,dxds+2\alpha\iint_{R(t)}v^{\alpha}\eta|\nabla\eta|\,dxds.
\end{aligned}
\end{equation}
By \eqref{ineqn-1} and \eqref{bd-estimate} we get \eqref{step1-ineqn} and the lemma follows.
\end{proof}

By Lemma \ref{lem-the-first-lemma-for-moser-iteration-i3}, an argument similar to the proof in Section 2 of \cite{Hs3}, and a compactness argument we have the following result.

\begin{prop}\label{Lp-L-infty-thm}
Let $n\geq 3$, $0<m\leq \frac{n-2}{n}$, $p>\frac{n(1-m)}{2}$, $0<\delta<\delta_0$, and let $f\in L^1(\partial\Omega\times(0,T))$ be such that $f\leq M$ for some constant $M>0$. Suppose $u\in C^{2,1}(\2{\Omega}_{\delta}\times (0,T))$ is a solution of \eqref{Neumann-holes-eqn} in $\Omega_{\delta}\times (0,T)$. Then for any $0<t_1<t_2<T$, $r_1\in (0,\delta_0/2)$, there exists constants $\theta>0$ and $C>0$ depending on $M$, $m$, and $n$ such that
\begin{equation*}
\|u\|_{L^{\infty}(E_{r_1}\times(t_2,T))}\leq C\left\{1+\int_{t_1}^T\int_{E_{2r_1}}u^p\,dxdt\right\}^{\frac{\theta}{p}}
\end{equation*}
where $E_{r_1}=\{x\in\Omega : \textbf{dist}(x,\partial\Omega)<r_1\}$.
\end{prop}

Similarly we have the following result.

\begin{prop}\label{Lp-L-infty-thm-1}
Let $n\ge 3$, $0<m\le \frac{n-2}{n}$, $p>\frac{n(1-m)}{2}$, $0<\delta<\delta_0$, and let $f\in L^1(\partial\Omega\times(0,T))$ and $g_i\in L^1(\partial B_{\delta}(a_i)\times (0,T))$ be such that $f\leq M$ and $g_i\leq M$ for all $i=1,\dots,i_0$ and some constant $M>0$. Suppose $u\in C^{2,1}(\2{\Omega}_{\delta}\times (0,T))$ is a solution of \eqref{Neumann-holes-eqn}  in $\Omega_{\delta}\times (0,T)$. Then for any $0<t_1<t_2<T$, there exist constants  $\theta>0$ and $C>0$ depending on $M$, $\delta$, $m$, and $n$ such that
\begin{equation*}
\|u\|_{L^{\infty}(\Omega_{\delta}\times(t_2,T])}\leq C\left(1+\iint_{\Omega_{\delta}\times(t_1,T]}u^p\,dxdt\right)^{\frac{\theta}{p}}.
\end{equation*}
\end{prop}

\begin{prop}\label{Lp-L-infty-thm-2}
Let $n\ge 3$, $0<m\le \frac{n-2}{n}$, $p>\frac{n(1-m)}{2}$, $0<\delta<\delta_0$, and let $f\in L^1(\partial\Omega\times(0,T))$ and $g_i\in L^1(\partial B_{\delta}(a_i)\times (0,T))$ be such that $f\leq M$ and $g_i\leq M$ for all $i=1,\dots,i_0$ and some constant $M>0$. Suppose $0\le u_0\in L^p(\Omega_{\delta})$ and $u\in C^{2,1}(\2{\Omega}_{\delta}\times (0,T))$ is a solution of \eqref{Neumann-holes-eqn} in $\Omega_{\delta}\times (0,T)$. Then for any $0<t_1<T$ there exist constants  
$\theta>0$ and $C>0$ depending on $M$, $\delta$, $m$, and $n$ such that
\begin{equation*}
\|u\|_{L^{\infty}(\Omega_{\delta}\times(t_1,T])}\leq C\left(1+\int_{\Omega_{\delta}}u_0^p\,dx\right)^{\frac{\theta}{p}}.
\end{equation*}
\end{prop}
\begin{proof}
Let $0<a\leq 1$ and $v=\max(u,a)$. As before by Lemma \ref{weak-w12-soln-lem} the function $v$ satisfies 
\eqref{weak-w12-soln2}. Since the function $v^{p-1}\in C(\2{\Omega}_{\delta}\times (0,T))$ and its first order weak derivatives belong to $L^2({\Omega}_{\delta}\times (0,T))$, by \eqref{weak-w12-soln2},
\begin{align}\label{v-integral-time-derivative}
\frac{\1}{\1 t}\left(\int_{\Omega_{\delta}}v^p\,dx\right)
=&p\int_{\Omega_{\delta}}v^{p-1}v_t\,dx\notag\\
\le&-p\int_{\Omega_{\delta}}\nabla v^{p-1}\cdot\nabla v^m\,dx
+p\int_{\1\Omega_{\delta}\cup(\cup_{i=1}^{i_0}\1 B_{\delta}(a_i))}v^{p-1}\frac{\1 v^m}{\1\nu}\,d\sigma\notag\\
\le &-p(p-1)m\int_{\Omega_{\delta}}v^{p+m-3}|\nabla v|^2\,dx
+Mp\int_{\1\Omega_{\delta}\cup(\cup_{i=1}^{i_0}\1 B_{\delta}(a_i))}v^{p-1}\,d\sigma.
\end{align}
By an argument similar to the proof of Lemma 1.1 of \cite{Hu1} there exists a constant $C_1>0$ such that,
\begin{align}\label{v-bdary-integral-ineqn}
\int_{\1\Omega_{\delta}\cup(\cup_{i=1}^{i_0}\1 B_{\delta}(a_i))}v^{p-1}\,d\sigma
\le&C_1\left(\int_{\Omega_{\delta}}|\nabla v^{p-1}|\,dx+\int_{\Omega_{\delta}}v^{p-1}\,dx\right)\notag\\
=&(p-1)C_1\int_{\Omega_{\delta}}v^{p-2}|\nabla v|\,dx+C_1\int_{\Omega_{\delta}}v^{p-1}\,dx\notag\\
\le&\frac{m(p-1)}{2M}\int_{\Omega_{\delta}}v^{p+m-3}|\nabla v|^2\,dx
+\frac{(p-1)MC_1^2}{2m}\int_{\Omega_{\delta}}v^{p-m-1}\,dx+C_1\int_{\Omega_{\delta}}v^{p-1}\,dx.
\end{align}
By \eqref{v-integral-time-derivative} and \eqref{v-bdary-integral-ineqn},
\begin{align*}
\frac{\1}{\1 t}\left(\int_{\Omega_{\delta}}v^p\,dx\right)
\le&C_2\left(\int_{\Omega_{\delta}}v^{p-m-1}\,dx+\int_{\Omega_{\delta}}v^{p-1}\,dx\right)\\
\le&\frac{2C_2}{a^m}\int_{\Omega_{\delta}}v^{p-1}\,dx\\
\le&\frac{2C_2|\Omega|^{1/p}}{a^m}\left(\int_{\Omega_{\delta}}v^p\,dx\right)^{\frac{p-1}{p}}
\end{align*}
for some constant $C_2>0$. Hence
\begin{equation}\label{u^p-integral-bd}
\left(\int_{\Omega_{\delta}}u^p(x,t)\,dx\right)^{\frac{1}{p}}\le\left(\int_{\Omega_{\delta}}v^p(x,t)\,dx\right)^{\frac{1}{p}}\le\left(\int_{\Omega_{\delta}}(u_0+a)^p\,dx\right)^{\frac{1}{p}}+2C_2p^{-1}a^{-m}|\Omega|^{1/p}t\quad\forall 0<t<T.
\end{equation}
By \eqref{u^p-integral-bd} and  Proposition \ref{Lp-L-infty-thm-1},  Proposition \ref{Lp-L-infty-thm-2} follows.
\end{proof}

\begin{lemma}\label{u-Lp-loc-bd-u0-Lp-loc-lemma}
Let $n\ge 3$, $0<m\le \frac{n-2}{n}$, $p>\frac{n(1-m)}{2}$, $0<\delta<\delta_2<\delta_1<\delta_0$, and $g_i\in L^1(\partial B_{\delta}(a_i)\times (0,T))$ for all $i=1,\dots,i_0$ and let $f\in L^1(\partial\Omega\times(0,T))$ be such that $f\leq M$ for some constant $M>0$. Suppose $0\le u_0\in L^p(\Omega_{\delta})$ and $u\in C^{2,1}(\2{\Omega}_{\delta}\times (0,T))$ is a solution of \eqref{Neumann-holes-eqn} in $\Omega_{\delta}\times (0,T)$. Then there exists a constant  $C>0$ depending on $M$, $\delta_1$, $\delta_2$, $m$, and $n$ such that
\begin{equation*}
\left(\int_{E_{\delta_0-\delta_1}}u^p(x,t)\,dx\right)^{\frac{1-m}{p}}
\le\left(\int_{E_{\delta_0-\delta_3}}(u_0+1)^p\,dx\right)^{\frac{1-m}{p}}+Ct\quad\forall 0<t<T
\end{equation*}
holds for any $0<\delta\le\delta_2$ where $\delta_3=(\delta_1+\delta_2)/2$ and $E_r=\{x\in\Omega:\mbox{dist}\, (x,\1\Omega)<r\}$ for any $0<r<\delta_0$.
\end{lemma}
\begin{proof}
We choose $\phi\in C^{\infty}(\2{\Omega})$, $0\le\phi\le 1$, such that $\phi (x)=1$ for all $x\in E_{\delta_0-\delta_1}$, $\phi (x)=0$ for all $x\not\in E_{\delta_0-\delta_3}$. Let $\eta=\phi^{\alpha}$ for some constant $\alpha>0$ to be chosen later and let $v=\max (u,1)$. By Lemma \ref{weak-w12-soln-lem} the function $v$ satisfies 
\eqref{weak-w12-soln2}. Since the function $v^{p-1}\eta^2\in C(\2{\Omega}_{\delta}\times (0,T))$ and its first order weak derivatives belong to $L^2({\Omega}_{\delta}\times (0,T))$, by \eqref{weak-w12-soln2},
\begin{align}\label{v-eta-integral-time-derivative}
\frac{\1}{\1 t}\left(\int_{\Omega_{\delta}}v^p\eta^2\,dx\right)
=&p\int_{\Omega_{\delta}}v^{p-1}\eta^2v_t\,dx\notag\\
\le&-p\int_{\Omega_{\delta}}\nabla (v^{p-1}\eta^2)\cdot\nabla v^m\,dx
+p\int_{\1\Omega}v^{p-1}\eta^2\frac{\1 v^m}{\1\nu}\,d\sigma\notag\\
\le&-p(p-1)m\int_{\Omega_{\delta}}v^{p+m-3}\eta^2|\nabla v|^2\,dx
+2pm\int_{\Omega_{\delta}}v^{p+m-2}\eta|\nabla\eta||\nabla v|\,dx\notag\\
&\qquad +pM\int_{\1\Omega}v^{p-1}\eta^2\,d\sigma.
\end{align}
Now
\begin{equation}\label{v-Lq-1}
2\int_{\Omega_{\delta}}v^{p+m-2}\eta|\nabla\eta||\nabla v|\,dx
\le\frac{(p-1)}{4}\int_{\Omega_{\delta}}v^{p+m-3}\eta^2|\nabla v|^2\,dx
+\frac{4}{p-1}\int_{\Omega_{\delta}}v^{p+m-1}|\nabla\eta|^2\,dx.
\end{equation}
By an argument similar to the proof of Lemma 1.1 of \cite{Hu1} there exists a constant $C_1>0$ such that,
\begin{align}\label{v-eta-bdary-integral-ineqn}
&\int_{\1\Omega}v^{p-1}\eta^2\,d\sigma\notag\\
&\qquad \le C_1\int_{\Omega_{\delta}}|\nabla (v^{p-1}\eta^2)|\,dx \notag\\
&\qquad  \le (p-1)C_1\int_{\Omega_{\delta}}v^{p-2}\eta^2|\nabla v|\,dx+2C_1\int_{\Omega_{\delta}}v^{p-1}\eta|\nabla\eta|\,dx\notag\\
&\qquad \le\frac{m(p-1)}{2M}\int_{\Omega_{\delta}}v^{p+m-3}\eta^2|\nabla v|^2\,dx
+\frac{(p-1)MC_1^2}{2m}\int_{\Omega_{\delta}}v^{p-m-1}\eta^2\,dx
+2C_1\int_{\Omega_{\delta}}v^{p-1}\eta |\nabla\eta|\,dx.
\end{align}
By \eqref{v-eta-integral-time-derivative}, \eqref{v-Lq-1} and \eqref{v-eta-bdary-integral-ineqn},
\begin{align}\label{v^q-time-derivative-ineqn}
&\frac{\1}{\1 t}\left(\int_{\Omega_{\delta}}v^p\eta^2\,dx\right)\notag\\
&\qquad \le C_2\left(\int_{\Omega_{\delta}}v^{p-m-1}\eta^2\,dx+\int_{\Omega_{\delta}}v^{p-1}\eta |\nabla\eta|\,dx
+\int_{\Omega_{\delta}}v^{p+m-1}|\nabla\eta|^2\,dx\right)\notag\\
&\qquad \le C_2\left[\int_{\Omega_{\delta}}v^{p-1}\eta^2\,dx
+\int_{\Omega_{\delta}}(v^p\eta^2)^{\frac{p-1}{p}} |\nabla\eta|\eta^{-1+\frac{2}{p}}\,dx
+\int_{\Omega_{\delta}}(v^p\eta^2)^{\frac{p+m-1}{p}}(|\nabla\eta|\eta^{-1+\frac{1-m}{p}})^2\,dx\right]
\end{align}
for some constant $C_2>0$. 
Now
\begin{equation}\label{eta-expression-bd1}
|\nabla\eta|\eta^{-1+\frac{2}{p}}=\alpha\phi^{\frac{2\alpha}{p}-1}|\nabla\phi|\in L^{\infty}\quad\mbox{ if }\alpha>\frac{p}{2}
\end{equation}
and
\begin{equation}\label{eta-expression-bd2}
|\nabla\eta|\eta^{-1+\frac{1-m}{p}}=\alpha\phi^{\frac{(1-m)}{p}\alpha-1}|\nabla\phi|\in L^{\infty}\quad\mbox{ if }
\alpha>\frac{p}{1-m}.
\end{equation}
We now choose $\alpha>\max \left(\frac{p}{2},\frac{p}{1-m}\right)$. Since
\begin{equation}
\int_{\Omega_{\delta}}v^p\eta^2\,dx\ge\int_{\Omega}\eta^2\,dx>0,
\end{equation}
by \eqref{v^q-time-derivative-ineqn}, \eqref{eta-expression-bd1} and
\eqref{eta-expression-bd2},
\begin{align}\label{v^q-time-derivative-ineqn2}
\frac{\1}{\1 t}\left(\int_{\Omega_{\delta}}v^p\eta^2\,dx\right)
\le&C_3\left[\int_{\Omega_{\delta}}v^{p-1}\eta^2\,dx
+\int_{\Omega_{\delta}}(v^p\eta^2)^{\frac{p-1}{p}} \,dx
+\int_{\Omega_{\delta}}(v^p\eta^2)^{\frac{p+m-1}{p}}\,dx\right]\notag\\
\le&C_4\left[\left(\int_{\Omega_{\delta}}v^p\eta^2\,dx\right)^{\frac{p-1}{p}}+\left(\int_{\Omega_{\delta}}v^p\eta^2\,dx\right)^{\frac{p+m-1}{p}}\right]\notag\\
\le&C_5\left(\int_{\Omega_{\delta}}v^p\eta^2\,dx\right)^{\frac{p+m-1}{p}}
\end{align}
for some constants $C_3>0$, $C_4>0$, $C_5>0$. Integrating \eqref{v^q-time-derivative-ineqn2},
\begin{equation*}
\left(\int_{\Omega_{\delta}}u^p(x,t)\eta^2(x)\,dx\right)^{\frac{1-m}{p}}\le\left(\int_{\Omega_{\delta}}v^p(x,t)\eta^2(x)\,dx\right)^{\frac{1-m}{p}}\le\left(\int_{\Omega_{\delta}}(u_0+1)^p\eta^2\,dx\right)^{\frac{1-m}{p}}+C_6t\quad\forall 0<t<T
\end{equation*}
for some constant $C_6>0$ and the lemma follows.
\end{proof}

By a similar argument we also have the following result.

\begin{prop}\label{Lp-L-infty-thm-3}
Let $n\ge 3$, $0<m\le \frac{n-2}{n}$, $p>\frac{n(1-m)}{2}$, $0<\delta<\delta_0$. Let $0\le u_0\in L^p(\Omega_{\delta})$, $g_i\in L^1(\partial B_{\delta}(a_i)\times (0,T))$ for all $i=1,2,\dots,i_0$, $f\in L^1(\partial\Omega\times(0,T))$  be such that $f\leq M$ for some constant $M>0$. Suppose $u\in C^{2,1}(\2{\Omega}_{\delta}\times (0,T))$ is a solution of \eqref{Neumann-holes-eqn} in $\Omega_{\delta}\times (0,T)$. Then for any $0<t_1<T$, $\delta<\delta_1<\delta_2\le\delta_0$, there exist constants  $\theta>0$ and $C>0$ depending on $M$, $m$, $n$, $t_1$, $\delta_1$, and
$\delta_2$ such that
\begin{equation*}
\|u\|_{L^{\infty}(\Omega_{\delta_2}\times(t_1,T])}\leq C\left\{1+\int_{\Omega_{\delta_1}}u_0^p\,dx\right\}^{\frac{\theta}{p}}.
\end{equation*}
and for any $0<t_1<T$, $R_2>R_1>0$ such that $B_{R_2}(x_1)\subset\Omega_{\delta}$, there exist constants  $\theta>0$ and $C>0$ depending on $m$, $n$, $t_1$, $R_1$ and $R_2$, but independent of $M$ such that
\begin{equation*}
\|u\|_{L^{\infty}(B_{R_1}(x_1)\times(t_1,T])}\le C\left\{1+\int_{B_{R_2}(x_1)}u_0^p\,dx\right\}^{\frac{\theta}{p}}.
\end{equation*}
\end{prop}

\section{Existence of solutions for the Neumann problem}
\setcounter{equation}{0}
\setcounter{thm}{0}

In this section we will prove the existence of solutions of the Neumann problem \eqref{Neumann-holes-eqn}. 
We first observe that by an argument similar to the proof of Proposition A.1 of \cite{BV1} we have the following lemma.

\begin{lemma}\label{Alexsandrov-reflection-lemma}
Let $n\ge 1$, $0<m<1$, and $0\le v_0(x)\in L^1(B_{5R}(x_0))$, $v_0\not\equiv 0$, such that $\text{supp }v_0\subset B_{R}(x_0)\subset\R^n$. Let $v$ be a weak solution of
\begin{equation}\label{Dirichlet-problem}
\left\{\begin{aligned}
v_t=&\Delta v^m\qquad\mbox{ in }B_{5R}(x_0)\times(0,T_{v_0})\\
v(x,t)=&0\qquad\quad\,\mbox{ on }\partial B_{5R}(x_0)\times(0,T_{v_0})\\
v(x,0)=&v_0(x)\quad\,\,\mbox{ in }B_{5R}(x_0)
\end{aligned}\right.
\end{equation}
where $T_{v_0}>0$ is the extinction time of $v$, then 
\begin{equation*}
v(x_1,t)\ge v(x_2,t)\quad\forall |x_1|\le R, 4R\le |x_2|\le 5R.
\end{equation*}
\end{lemma}

By Lemma \ref{Alexsandrov-reflection-lemma} and an argument similar to the proof in section 1 of \cite{BV2} we have the following result.

\begin{lemma}\label{Dirichlet-soln-positive-lemma}(cf. (1.18) and (1.27) of \cite{BV2})
Let $n$, $m$, and $v_0$, $v$, and $T_{v_0}$, be as in Lemma \ref{Alexsandrov-reflection-lemma}.Then there exist constants $C_1>0$, $C_2>0$, such that
\begin{equation*}
T_{v_0}\ge C_1R^{2-n(1-m)}\left(\int_{B_R(x_0)}v_0\,dx\right)^{1-m}
\end{equation*}
and
\begin{equation*}
v^m(x,t)\ge C_2R^{2-n}\|v_0\|_{L^1(B_{R}(x_0))}T_{v_0}^{-\frac{1}{1-m}}t^{\frac{m}{1-m}}\quad \forall |x|\le R, t\in (0,t_{\ast}]
\end{equation*}
where
\begin{equation*}
t_{\ast}=\frac{C_1}{2}R^{2-n(1-m)}\left(\int_{B_R(x_0)}v_0\,dx\right)^{1-m}.
\end{equation*}
\end{lemma}

\begin{lemma}\label{positivity-lemma}
Let $n\ge 3$, $0<m\leq\frac{n-2}{n}$ and $T>t_0>0$. Suppose $0\le u\in C(\Omega\times(t_0,T])$
satisfies \eqref{fast-diffusion-eqn} in $\mathcal{D}(\Omega\times (t_0,T))$ and
\begin{equation}\label{u-integral-positive}
\int_{\Omega}u(x,T)\,dx>0.
\end{equation}
Then
\begin{equation*}
u(x,T)>0 \qquad \forall x\in\Omega.
\end{equation*}
\end{lemma}
\begin{proof}
Let
\begin{equation*}
D(T)=\left\{x\in\Omega :u(x,T)>0\right\}.
\end{equation*}
Since by \eqref{u-integral-positive} there exists a point $x(T)\in \Omega$ such that $u(x(T),T)>0$, $D(T)\ne\phi$. Suppose that
$D(T)\neq\Omega$. Then there exist a point $x_0\in\Omega\cap\1 D(T)$ such that
\begin{equation}\label{u=0-at-x0}
u(x_0,T)=0.
\end{equation}
We choose a constant $R>0$ such that $B_{5R}(x_0)\subset\Omega$. We claim that
\begin{equation}\label{eq-properties-from-assumption-of-u-epsilon-at-T}
\int_{B_{R/2}(x_0)}u(x,T)\,dx>0.
\end{equation}
Suppose not. Then
\begin{equation*}
\int_{B_{R/2}(x_0)}u(x,T)\,dx=0\quad\Rightarrow\quad u(x,T)=0\quad\forall |x-x_0|\le\frac{R}{2}
\end{equation*}
which contradicts the fact that $x_0\in\partial D(T)$. Hence \eqref{eq-properties-from-assumption-of-u-epsilon-at-T} holds.
Let $\psi\in C_0^{\infty}(B_{5R}(x_0))$, $0\le\psi\le 1$, be such that $\psi(x)=1$ 
for all $x\in B_{R/2}(x_0)$ and $\psi(x)=0$ for all $x\in B_{5R}(x_0)\bs B_R(x_0)$. By the proof of Lemma 3.1 of \cite{HP}
there exist constants $\alpha>1$ and $C_R>0$ such that
\begin{align}\label{eq-the-difference-between-1-m-norm-and-T-s234}
&\left|\frac{\1}{\1 t}\left(\int_{B_{5R}(x_0)}u\psi^{\alpha}\,dx\right)\right|\le C_R\left(\int_{B_{5R}(x_0)}u\psi^{\alpha}\,dx\right)^m\quad\forall 0<t<T\notag\\
\Rightarrow\quad&\left|\left(\int_{B_{5R}(x_0)}u(x,T)\psi^{\alpha}(x)\,dx\right)^{1-m}-\left(\int_{B_{5R}(x_0)}u(x,s)\psi^{\alpha}(x)\,dx\right)^{1-m}\right|\le (1-m)C_R(T-s) \quad \forall T>s>0.
\end{align}
Let
\begin{equation*}
\3_1=\frac{1-2^{m-1}}{(1-m)C_{R_1}}\left(\int_{B_{5R}(x_0)}u(x,T)\psi^{\alpha}(x)\,dx\right)^{1-m}.
\end{equation*}
By \eqref{eq-the-difference-between-1-m-norm-and-T-s234} for any $0<T-s\leq\3_1$
\begin{align}\label{eq-aligned-uniform-positivity-of-mass-of-u-near-T-0--}
\int_{B_R(x_0)}u(x,s)\,dx\ge&\int_{B_{5R}(x_0)}u(x,s)\psi^{\alpha}(x)\,dx\notag\\
\ge&\left[\left(\int_{B_{5R}(x_0)}u(x,T)\psi^{\alpha}(x)\,dx\right)^{1-m}-(1-m)C_R(T-s)\right]^{\frac{1}{1-m}}\notag\\
\ge&\frac{1}{2}\int_{B_{5R}(x_0)}u(x,T)\psi^{\alpha}(x)\,dx\notag\\
\ge&\frac{1}{2}\int_{B_{R/2}(x_0)}u(x,T)\,dx.
\end{align}
Let $v_{0,\tau}(x)=u(x,T-\tau)\chi_{B_R(x_0)}$. By the discussion on P.537 of \cite{BV2} there exists a unique weak minimal solution $v^{\tau}$ of \eqref{Dirichlet-problem} with initial value $v_{0,\tau}$. Let $T_{v_{0,\tau}}$ be the extinction time of $v^{\tau}$. Then by \eqref{eq-properties-from-assumption-of-u-epsilon-at-T}, \eqref{eq-aligned-uniform-positivity-of-mass-of-u-near-T-0--} and Lemma \ref{Dirichlet-soln-positive-lemma},
\begin{equation*}
T_{v_{0,\tau}}\geq c_0>0\quad\forall 0<\tau\le\3_1.
\end{equation*}
where 
\begin{equation*}
c_0=C_1R^{2-n(1-m)}\left(\frac{1}{2}\int_{B_{R/2}(x_0)}u(x,T)\,dx\right)^{1-m} \qquad \mbox{and} \qquad  \mbox{$C_1>0$ is as in Lemma \ref{Dirichlet-soln-positive-lemma}.}
\end{equation*}
Let $\tau_1=\min (c_0,\3_1)/2$. By Lemma \ref{Dirichlet-soln-positive-lemma} there exists a constant $c_1>0$ such that
\begin{equation}\label{min-soln-bded-below}
v^{\tau_1}(x,\tau_1)\ge c_1\quad\forall |x|\le R.
\end{equation}
Since $v^{\tau_1}$ is the unique weak minimal solution  of \eqref{Dirichlet-problem} with initial value $v_{0,\tau_1}$,
by the maximum principle,
\begin{equation}\label{u-lower-bd}
u^m(x,T)\ge (v^{\tau_1})^m(x,\tau_1)>0\quad\forall |x|\le R.
\end{equation}
By \eqref{min-soln-bded-below} and \eqref{u-lower-bd} we get $u^m(x_0,T)\ge c_1>0$. This contradicts \eqref{u=0-at-x0}. Hence $D(T)=\Omega$ and the lemma follows.
\end{proof}

\begin{lemma}\label{Aronson-Benilan-ineqn-lemma}
Let $n\ge 1$, $0<m<1$, $0<\delta<\delta_0$, $0<u_0\in C^2(\2{\Omega}_{\delta})$.
Let $f\in C^{\infty}(\1\Omega\times [0,T))$ and $g_i\in C^{\infty}(\1 B_{\delta}(a_i)\times [0,T))$ for all $i=1,\cdots,i_0$ be such that $f$, $g_i$, are nonnegative monotone decreasing functions of $t\in (0,T)$, and $f\le M$, $g_i\le M$, for all $i=1,\cdots,i_0$ and some constant $M>0$. Suppose $u\in C^{2,1}(\2{\Omega}_{\delta}\times [0,T))$ is a positive solution of \eqref{Neumann-holes-eqn}. Then $u$ satisfies \eqref{Aronson-Benilan-ineqn} in $\Omega_{\delta}\times (0,T)$.
\end{lemma}
\begin{proof}
Let
\begin{equation*}
q=\frac{u_t}{u}, \qquad \3=\min_{\overline{\Omega}_{\delta}}u_0, \qquad a=\frac{\3}{(1-m)\left(\|\La u_0^m\|_{L^{\infty}(\Omega_{\delta})}+1\right)}
\end{equation*}
and
\begin{equation*}
\overline{q}=q-\frac{1}{(1-m)(a+t)}.
\end{equation*}
By direct computation, 
\begin{equation}\label{q-eqn}
q_t=mu^{m-1}\La q+2mu^{m-2}\nabla u\cdot\nabla q+(m-1)q^2 \qquad \mbox{in $\Omega_{\delta}\times(0,T)$}.
\end{equation}
Then $\2{q}$ satisfies 
\begin{equation}\label{q-bar-eqn}
\begin{cases}
\begin{aligned}
&\overline{q}_t=mu^{m-1}\La\overline{q}+2mu^{m-2}\nabla u\cdot\nabla\overline{q}-(1-m)\overline{q}\left(\overline{q}+\frac{2}{(1-m)(a+t)}\right)\qquad \mbox{in $\Omega_{\delta}\times(0,T)$}\\
&\overline{q}(x,0)\leq 0 \qquad \qquad \qquad \qquad \mbox{on $\Omega_{\delta}$}.
\end{aligned}
\end{cases}
\end{equation}
Since
\begin{align}
&f_t=\frac{\partial}{\partial t}\left(\frac{\partial u^m}{\partial \nu}\right)=mu^{m-1}u_{\nu t}-m(1-m)u^{m-2}u_{\nu}u_t\quad
\quad\forall (x,t)\in\partial\Omega\times(0,T)\notag\\
&\qquad \Rightarrow\quad u_{\nu t}=\frac{f_tu^{1-m}}{m}+(1-m)\frac{u_{\nu}u_t}{u}\qquad\qquad\qquad\,
\forall (x,t)\in\partial\Omega\times(0,T),
\end{align}
we have
\begin{align}\label{eq-of-overline-q-sub-nu-on-partial-Omega-with-f-t-and-u-and-a}
\frac{\partial q}{\partial\nu}&=\frac{u_{\nu t}}{u}-\frac{u_{\nu}u_t}{u^2}=\frac{f_t}{mu^{m}}-\frac{qf}{u^m}\qquad\quad\,\mbox{ on }\partial\Omega\times(0,T)\notag\\
&=\frac{f_t}{mu^{m}}-\frac{\2{q}f}{u^m}-\frac{f}{(1-m)(a+t)u^m}\quad \mbox{on $\partial\Omega\times(0,T)$}.
\end{align}
Similarly
\begin{equation*}
\frac{\partial q}{\partial\nu}=\frac{g_{i,t}}{mu^{m}}-\frac{\2{q}g_i}{u^m}-\frac{g_i}{(1-m)(a+t)u^m}\qquad \mbox{ on }\partial B_{\delta}(a_i)\times(0,T)\quad\forall i=1,\cdots,i_0.
\end{equation*}
Let $0<T_1<T$ and suppose $\overline{q}$ attains a positive maximum on $\overline{\Omega}_{\delta}\times(0,T_1]$ at $(x_0,t_0)$
for some $x_0\in\2{\Omega}_{\delta}$, $0\le t_0\le T_1$. Suppose $x_0\in\Omega_{\delta}$ and $t_0>0$. Then
\begin{equation}\label{q-bar-ineqns-at-max}
\overline{q}_t\geq 0, \qquad \nabla\overline{q}=0 \qquad \mbox{and} \qquad \La\overline{q}\leq 0  \qquad \mbox{at $(x_0,t_0)$}.
\end{equation}
Hence by \eqref{q-bar-eqn} and \eqref{q-bar-ineqns-at-max},
\begin{align}
0\le\overline{q}_t=&mu^{m-1}\La\overline{q}+2mu^{m-2}\nabla u\cdot\nabla\overline{q}-(1-m)\overline{q}\left(\overline{q}+\frac{2}{(1+m)(a+t)}\right)\notag\\
\le&-(1-m)\overline{q}\left(\overline{q}+\frac{2}{(1+m)(a+t)}\right)<0\qquad \qquad \mbox{at $(x_0,t_0)$}.\notag
\end{align}
Contradiction arises. Hence either $x_0\in\partial\Omega_{\delta}$ or $t_0=0$. Suppose $x_0\in\partial\Omega_{\delta}$ and $t_0>0$. 
Without loss of generality we may assume that $x_0\in\partial\Omega$. By the strong maximum principle, 
\begin{equation}\label{q-bar-bdary-derivative-+ve}
\frac{\1\2{q}}{\1\nu}(x_0,t_0)>0.
\end{equation}
Then by \eqref{eq-of-overline-q-sub-nu-on-partial-Omega-with-f-t-and-u-and-a} and \eqref{q-bar-bdary-derivative-+ve},
\begin{align*}
&0<\frac{f_t}{mu^{m}}-\frac{\2{q}f}{u^m}-\frac{f}{(1-m)(a+t)u^m}\le-\frac{\2{q}f}{u^m}\quad\mbox{ at }(x_0,t_0)\\
&\Rightarrow\quad \2{q}f<0\quad\mbox{ at }(x_0,t_0)
\end{align*}
Since $\2{q}(x_0,t_0)>0$ and $f(x_0,t_0)\ge 0$, contradiction arises. Hence $t_0=0$. Since $\overline{q}(x,0)\leq 0$ on 
$\Omega_{\delta}$,  
\begin{align*}
&\overline{q}(x,t)\le 0\quad\forall x\in\Omega_{\delta}, 0<t<T_1\notag\\
&\Rightarrow\quad\overline{q}(x,t)\le 0 \quad\forall x\in\Omega_{\delta}, 0<t<T\quad\mbox{ as }T_1\to T\\
&\Rightarrow\quad u_t\le\frac{u}{(1-m)(a+t)}\le\frac{u}{(1-m)t}\quad\mbox{ in }\Omega_{\delta}\times (0,T)
\end{align*}
and the lemma follows.
\end{proof}

We are now ready for the proof of Theorem \ref{Neumann-holes-problem-existence-thm}.

\begin{proof}[Proof of Theorem \ref{Neumann-holes-problem-existence-thm}]
By Lemma \ref{comparsion-lemma1} the solution of \eqref{Neumann-holes-eqn} is unique. Hence it remains to prove the existence of solution of \eqref{Neumann-holes-eqn}. We will use a modification of the proof of Lemma 2.1 of \cite{Hs2} to prove this theorem.  
We divide the proof into three cases.\\
\textbf{Case 1}: $0<u_0\in C^{\infty}(\overline{\Omega}_{\delta})$, $f\in C^{\infty}(\1\Omega\times [0,\infty))$, $g_i\in C^{\infty}(\1 B_{\delta}(a_i)\times [0,\infty))$, such that $g_{i}(x,t)=\alpha_i$ for all $(x,t)\in\partial B_{\delta}(a_i)\times[0,\delta')$, $i=1,\cdots,i_0$, and
$f(x,t)=0$  for all $(x,t)\in\partial \Omega\times[0,\delta')$ for some constants $\delta'>0$ and $\alpha_1,\cdots,\alpha_{i_0}\in\R^+$, respectively.

Let $\3_0=\inf_{\overline{\Omega}_{\delta}}u_0$. Then $\3_0>0$. We first choose a sequence of functions $\{\phi_j\}_{j=1}^{\infty}\subset C^{\infty}(\Omega)$, $0\le\phi_j\le 1$ for all $j\in\mathbb{Z}^+$, such that
\begin{equation*}
\phi_j(x)=\begin{cases}
1\qquad \mbox{if $\textbf{dist}(x,\partial\Omega)\leq \frac{\delta_0-\delta}{3(j+1)}$}\\
0\qquad \mbox{if $\textbf{dist}(x,\partial\Omega)> \frac{\delta_0-\delta}{3j}$}
\end{cases}
\end{equation*}
for all $j\in\mathbb{Z}^+$ and a sequence of functions $\{\psi_j\}_{j=1}^{\infty}\subset C_0^{\infty}(\R^N)$, $0\le\psi_j\le 1$ for all $j\in\mathbb{Z}^+$, such that
\begin{equation*}
\psi_j(x)=\begin{cases}
1\qquad\mbox{ if }|x|\le\delta+\frac{\delta_0-\delta}{3(j+1)}\\
0\qquad\mbox{ if }|x|\ge\delta+\frac{\delta_0-\delta}{3j}
\end{cases}
\end{equation*}
for all $j\in\mathbb{Z}^+$. Then $0\le\phi_{j+1}\le\phi_{j}\leq 1$ and $0\le\psi_{j+1}\le\psi_{j}\le 1$ for any $j\in\Z^+$. For any $j\in\mathbb{Z}^+$ and $x\in\Omega$ let
\begin{equation*}
u_{0,j}(x)=u_{0}(x)\left(1-\phi_j(x)-\sum_{i=1}^{i_0}\psi_j(x-a_i)\right)+\delta^{\frac{2}{m}}\sum_{i=1}^{i_0}\frac{\alpha_i^{\frac{1}{m}}}{|x-a_i|^{\frac{1}{m}}}\cdot\psi_{j}(x-a_i)+\3_0\phi_j
\end{equation*}
and
\begin{equation*}
\3_1=\min\left\{\3_0,\left(\frac{\delta^2\alpha}{\textbf{diam}\,\Omega}\right)^{\frac{1}{m}}\right\}, \qquad \mbox{where $\alpha=\min\left\{\alpha_1,\cdots,\alpha_{i_0}\right\}$}.
\end{equation*}
Then
\begin{equation}\label{u0j-upper-lower-bd}
\3_1\le u_{0,j} \le u_{0}+C_1 \quad\mbox{ in }\Omega_{\delta}\quad \forall j\in\mathbb{Z}^+\qquad\mbox{and}\qquad\|u_{0,j}-u_0\|_{L^{1}(\Omega_{\delta})}\to 0\quad\mbox{ as }j\to\infty
\end{equation}
for some constant $C_1>0$ depending on $m$, $\alpha_1,\cdots,\alpha_{i_0}$, $\delta$ and
\begin{equation*}
\begin{cases}
\begin{aligned}
\frac{\partial u_{0,j}^m}{\partial\nu}=&0\quad\,\mbox{ on }\partial\Omega\quad\forall j\in\Z^+\\
\frac{\partial u_{0,j}^m}{\partial\nu}=&\alpha_i\quad\mbox{ on }\partial B_{\delta}(a_i)\quad\forall
j\in\Z^+,i=1,2,\cdots,i_0.
\end{aligned}
\end{cases}
\end{equation*}
We will now use an argument similar to the proof of Theorem 3.5 of \cite{Hu1} to prove the existence of a solution of 
\eqref{Neumann-holes-eqn}  in $\Omega_{\delta}\times (0,\infty)$ with initial value $u_{0,j}$. Let $n_1=\3_1^m$ and 
$n_2=(\|u_0\|_{L^{\infty}}+C_1)^m$. We choose a monotone decreasing function $H\in C^{\infty}(\R)$, $H>0$, such that
$H(s)=ms^{1-\frac{1}{m}}$ for $n_1/2\le s\le 2n_2$, $H(s)=m(n_1/4)^{1-\frac{1}{m}}$ for $s\le n_1/4$, $H(s)=m(4n_1/4)^{1-\frac{1}{m}}$ for $s\ge 4n_2$. Then by standard theory for non-degerenate parabolic equation \cite{LSU} such that the problem
\begin{equation*}
\left\{\begin{aligned}
v_t=&H(v)\Delta v\quad\,\,\,\mbox{ in }\Omega_{\delta}\times (0,\infty)\\
\frac{\1 v}{\1\nu}=&f\qquad\qquad\mbox{ on }\1\Omega\times (0,\infty)\\
\frac{\1 v}{\1\nu}=&g_i\qquad\qquad\mbox{ on }\1 B_{\delta}(a_i)\times (0,\infty)\quad\forall i=1,\cdots,i_0\\
v(x,0)=&u_{0,j}(x,0)^m\quad\mbox{ in }\Omega_{\delta}
\end{aligned}\right.
\end{equation*}
has a classical solution $v_j\in C^{2,1}(\overline{\Omega}_{\delta}\times (0,\infty))$. By the maximum principle for non-degerenate parabolic equation \cite{LSU}, $n_1\le v_j\le n_2$ on $\overline{\Omega}_{\delta}\times (0,\infty)$. Hence $H(v_j)=mv_j^{1-\frac{1}{m}}$ and thus the function $u_j=v_j^{\frac{1}{m}}\in C^{2,1}(\overline{\Omega}_{\delta}\times (0,\infty))$ is a solution of \eqref{Neumann-holes-eqn}  in $\Omega_{\delta}\times (0,\infty)$ with initial value $u_{0,j}$ such that
\begin{equation}\label{uj-lower-bd}
u_j\ge\3_1\quad\mbox{ in }\overline{\Omega}_{\delta}\times[0,\infty).
\end{equation}

Since
\begin{equation*}
\frac{\1}{\partial t}\left(\int_{\Omega_{\delta}}u_j\,dx\right)=\int_{\Omega_{\delta}}\La u_j^m\,dx
=\int_{\partial\Omega}f\,d\sigma+\sum_{i=1}^{i_0}\int_{\partial B_{\delta}(a_i)}g_i\,d\sigma\quad \forall t>0, j\in\Z^+,
\end{equation*}
integrating over $t$ we have
\begin{equation}\label{eq-integration-form-from-main-construction-equation-1-with-j-1}
\int_{\Omega_{\delta}}u_j(x,t)\,dx=\int_{\Omega_{\delta}}u_{0,j}\,dx+\int_0^t\int_{\partial\Omega}fd\sigma ds+\sum_{i=1}^{i_0}\int_0^t\int_{\partial B_{\delta}(a_i)}g_i\,d\sigma ds \qquad t>0.
\end{equation}

We will now show that $u_j$ converges to a solution $u$ of \eqref{Neumann-holes-eqn} as $j\to \infty$. Let $t_2>t_1>0$. Then by \eqref{u0j-upper-lower-bd} and  Proposition \ref{Lp-L-infty-thm-2} there exists a constant $C_2>0$ such that
\begin{equation}\label{uj-uniform-upper-bd}
u_j\le C_2 \qquad \forall x\in\overline{\Omega}_{\delta},t_1\le t\le t_2,j\in\Z^+.
\end{equation}
Hence by \eqref{uj-lower-bd} and \eqref{uj-uniform-upper-bd} the equation \eqref{fast-diffusion-eqn} for the sequence $\{u_j\}_{j=1}^{\infty}$ is uniformly parabolic on $\overline{\Omega}_{\delta}\times [t_1,t_2]$ for any $t_2>t_1>0$. Thus by the Schauder estimates \cite{LSU} (Theorem 3.1 and Theorem 5.4 in chapter 5 of \cite{LSU}), for any $t_2>t_1>0$,
\begin{equation*}
\sup_{\overline{\Omega}_{\delta}\times[t_1,t_2]}(|\nabla u_j|+|\1^2_{x_kx_l}u_j|+|u_{j,t}|)\le C_3\quad\forall k,l=1,\cdots,n, j\in\Z^+
\end{equation*}
and
\begin{equation*}
\sup_{\substack{y,y'\in\overline{\Omega}_{\delta}\\s,s'\in [t_1,t_2]}}\frac{|\1^2_{x_kx_l}u_j(y,s)-\1^2_{x_kx_l}u_j(y',s')|}{|y-y'|^{\alpha}}+\sup_{\substack{y,y'\in\overline{\Omega}_{\delta}\\s,s'\in [t_1,t_2]}}\frac{|u_{j,t}(y,s)-u_{j,t}(y',s')|}{|s-s'|^{\frac{\alpha}{2}}}\le C_4\quad\forall k,l=1,\cdots,n, j\in\Z^+
\end{equation*}
for some constants $C_3>0$, $C_4>0$, $0<\alpha<1$.
Hence the sequence $\{u_{j}\}_{j=1}^{\infty}$ is uniformly H\"older continuous in $C^{2,1}(\overline{\Omega}_{\delta}\times[t_1,t_2])$ for any $t_2>t_1>0$. 
By the Ascoli Theorem and a diagonalization argument the sequence $\{u_{j}\}_{j=1}^{\infty}$ has a subsequence which we may assume without loss of generality to be the sequence itself that converges uniformly on every compact subset of $\overline{\Omega}_{\delta}\times (0,\infty)$ to a solution $u\in C^{2,1}(\overline{\Omega}_{\delta}\times(0,\infty))$ of \eqref{fast-diffusion-eqn} in $\overline{\Omega}_{\delta}\times(0,\infty)$ which satisfies
\begin{equation*}
\frac{\1 u^m}{\1\nu}=f\quad\mbox{ on }\1\Omega\times (0,\infty)\qquad\mbox{ and }\qquad\frac{\1 u^m}{\1\nu}=g_i\quad\mbox{ on }\1 B_{\delta}(a_i)\times (0,\infty)\quad\forall i=1,\dots,i_0
\end{equation*}
as $j\to\infty$. Letting $j\to\infty$ in \eqref{eq-integration-form-from-main-construction-equation-1-with-j-1}, $u$ satisfies \eqref{mass-formulua-holes}.

It remains to show that $u$ has initial value $u_0$. For any $\psi\in C^{\infty}_0(\Omega_{\delta})$, by \eqref{eq-integration-form-from-main-construction-equation-1-with-j-1},
\begin{equation*}
\begin{aligned}
&\left|\int_{\Omega_{\delta}}u_j(x,t)\psi(x)\,dx-\int_{\Omega_{\delta}}u_{0,j}(x)\psi(x)\,dx\right|\\
&\qquad \qquad =\left|\int_0^t\int_{\Omega_{\delta}}u_{j,t}(x,s)\psi(x)\,dxds\right|=\left|\int_0^t\int_{\Omega_{\delta}}u_{j}^m(x,s)\La\psi(x)\,dxds\right|\\
&\qquad \qquad \le\left|\Omega\right|^{1-m}\left\|\La\psi\right\|_{L^{\infty}}\int_{0}^{t}\left(\int_{\Omega_{\delta}}u_{j}(x,s)\,dx)\right)^{m}\,ds\\
&\qquad \qquad \le\left|\Omega\right|^{1-m}\left\|\La\psi\right\|_{L^{\infty}}\left(\int_{\Omega_{\delta}}u_{0,j}(x)\,dx+\int_0^t\int_{\partial\Omega}f\,d\sigma ds+\sum_{i=1}^{i_0}\int_0^t\int_{\partial B_{\delta}(a_i)}g_i\,d\sigma ds\right)^m t\quad\forall t>0.
\end{aligned}
\end{equation*}
Letting $j\to\infty$,
\begin{align}\label{u-u0-L1-bd}
&\left|\int_{\Omega_{\delta}}u(x,t)\psi(x)\,dx-\int_{\Omega_{\delta}}u_0(x)\psi(x)\,dx\right|\notag\\
&\qquad \qquad \le\left|\Omega\right|^{1-m}\left\|\La\psi\right\|_{L^{\infty}}\left(\int_{\Omega_{\delta}}u_0(x)\,dx+\int_0^t\int_{\partial\Omega}f\,d\sigma ds+\sum_{i=1}^{i_0}\int_0^t\int_{\partial B_{\delta}(a_i)}g_i\,d\sigma ds\right)^m t
\quad\forall t>0.
\end{align}
Letting $t\to 0$ in \eqref{u-u0-L1-bd},
\begin{equation}\label{u-converge-u0-weakly}
\lim_{t\to 0}\int_{\Omega_{\delta}}u(x,t)\psi(x)\,dx=\int_{\Omega_{\delta}}u_0(x)\psi(x)\,dx\qquad\forall\psi\in C^{\infty}_0(\Omega).
\end{equation}
By \eqref{u-converge-u0-weakly}, any sequence $\{t_k\}_{k=1}^{\infty}$ converging to $0$ as $k\to\infty$ will have a convergent subsequence $\{t_{k_l}\}_{l=1}^{\infty}$ such that $u(x,t_{k_l})$ converges to $u_0(x)$ a.e. $x\in\Omega_{\delta}$ as $l\to\infty$. By the Lebesgue Dominated Convergence Theorem
\begin{equation*}
\lim_{l\to\infty}\int_{\Omega_{\delta}}|u(x,t_{k_l})-u_0(x)|\,dx=0
\end{equation*}
Since the sequence $\{t_k\}_{k=1}^{\infty}$ is arbitrary, $u$ satisfies \eqref{u-L1-converge-u0-as-t-goto-0}.

\noindent\textbf{Case 2:} $0\leq u_0\in L^{\infty}\left(\overline{\Omega}_{\delta}\right)$.

\noindent We choose a sequence of functions $\{u_{0,j}\}_{j=1}^{\infty}\subset C^{\infty}(\overline{\Omega}_{\delta})$ such that $b_j:=\min_{\Omega_{\delta}}u_{0,j}>0$ on $\overline{\Omega}_{\delta}$ for all $j\in\Z^+$ and
\begin{equation}\label{eq-cases-aligned-approximations-to-u-by-u-j-08}
\begin{cases}
\begin{aligned}
&\|u_{0,j}-u_0\|_{L^1(\Omega_{\delta})}\to 0 \qquad  \qquad \mbox{as $j\to\infty$}\\
&\qquad u_{0,j}(x)\to u_0(x) \qquad \qquad \mbox{a.e. $x\in\Omega$} \quad \mbox{ as }j\to\infty\\
&\|u_{0,j}\|_{L^{\infty}(\Omega_{\delta})}\leq \|u_0\|_{L^{\infty}(\Omega_{\delta})}+\frac{1}{j}\quad\forall j\in\mathbb{Z}^+.
\end{aligned}
\end{cases}
\end{equation}
For each $i=1,\cdots,i_0$, we choose a sequence of positive functions $\{g_{i,j}\}_{j=1}^{\infty}\subset C^{\infty}(\partial B_{\delta}(a_i)\times[0,\infty))$  satisfying
\begin{equation}\label{eq-cases-aligned-construction-of-g-i-j-converging-to-g-i-08}
\begin{cases}
\begin{aligned}
g_{i,j}&\to g_i \qquad \qquad\qquad\qquad\qquad\mbox{in $L^1_{loc}(\partial B_{\delta}(a_i)\times[0,\infty))$} \qquad \mbox{as $j\to\infty$}\qquad \forall i=1,\cdots, i_0\\
g_{i,j}(x,t)&\to g_{i}(x,t) \qquad\qquad\qquad\qquad\mbox{a.e. $(x,t)\in\partial B_{\delta}(a_i)\times[0,\infty)$} \quad \mbox{as $j\to\infty$}\quad\,\,\, \forall i=1,\cdots, i_0\\
g_{i,j}(x,t)&\le \|g_i\|_{L^{\infty}(\1 B_{\delta}(a_i)\times [0,T])}+1\qquad\mbox{ on }\partial B_{\delta}(a_i)\times[0,T)\qquad\qquad\qquad\forall j\in\mathbb{Z}^+,\, i=1,\cdots, i_0, T>0\\
g_{i,j}&=\alpha_{i,j}\qquad \qquad\qquad\qquad\quad \mbox{ on }\partial B_{\delta}(a_i)\times  [0,1/j)\qquad\qquad\qquad\qquad\forall j\in\mathbb{Z}^+,\, i=1,\cdots, i_0
\end{aligned}
\end{cases}
\end{equation}
for some positive constants $\alpha_{i,j}$ and choose a sequence of nonnegative functions $\{f_j\}_{j=1}^{\infty}\subset C^{\infty}\left(\partial\Omega\times [0,\infty)\right)$ satisfying
\begin{equation}\label{eq-cases-aligned-construction-of-f-i-converging-to-g-i-08}
\begin{cases}
\begin{aligned}
f_j&\to f \qquad \qquad\qquad \quad\,\mbox{ in }L^1_{loc}(\partial\Omega\times[0,\infty)) \quad \mbox{as $j\to\infty$}\\
f_j(x,t)&\to f(x,t) \qquad\qquad\quad\,\, \,\mbox{a.e. $(x,t)\in\partial\Omega\times[0,\infty)$} \quad \mbox{as $j\to\infty$}\\
f_j(x,t)&\le\|f\|_{L^{\infty}(\1\Omega\times [0,T])}+1\quad\mbox{on }\partial \Omega_{\delta}\times[0,T)\quad
\forall j\in\mathbb{Z}^+, T>0\\
f_j&= 0 \qquad\qquad \qquad\quad \,\,\,\mbox{ on }\partial\Omega\times [0,1/j)\quad \forall j\in\mathbb{Z}^+.
\end{aligned}
\end{cases}
\end{equation}
If $f\equiv 0$ on $\1\Omega\times (0,\infty)$ and $g_i$ are nonnegative monotone decreasing function of $t\in (0,\infty)$ for $i=1,\dots,i_0$, then we choose $f_j\equiv 0$ on $\1\Omega\times (0,\infty)$ for all $j\in\Z^+$ and the functions $g_{i,j}$ such that they are positive monotone decreasing functions of $t\in (0,\infty)$ for $i=1,\dots,i_0$ and $j\in\Z^+$. 
Then by case 1 for any $j\in\Z^+$ there exists a solution $u_j\in C^{2,1}(\2{\Omega}_{\delta}\times(0,\infty))$ of \eqref{Neumann-holes-eqn} in $\Omega_{\delta}\times(0,\infty)$ with initial value $u_{0,j}$ that satisfies
\begin{equation}\label{eq-integration-form-from-main-construction-equation-1-with-j-2}
\int_{\Omega_{\delta}}u_j(x,t)\,dx=\int_{\Omega_{\delta}}u_{0,j}\,dx+\iint_{\partial\Omega\times(0,t)}f_j\,d\sigma ds
+\sum_{i=1}^{i_0}\iint_{\partial B_{\delta}(a_i)\times(0,t)}g_{i,j}\,d\sigma ds \qquad \forall t>0,j\in\Z^+.
\end{equation}
Let $t_2>t_1>0$. Then by \eqref{eq-cases-aligned-approximations-to-u-by-u-j-08} and  Proposition \ref{Lp-L-infty-thm-2}, there exists a constant $C>0$ such that \eqref{uj-uniform-upper-bd} holds. Since the constant function $b_j$ is a subsolution of \eqref{Neumann-holes-eqn} in $\Omega_{\delta}\times(0,\infty)$ with $u_0=b_j$. By Lemma \ref{comparsion-lemma1},
\begin{equation}\label{uj-uniform-lower-bd}
u_j(x,t)\ge b_j>0\quad\forall x\in\Omega_{\delta}, t\ge 0. 
\end{equation}
Then by \eqref{uj-uniform-upper-bd} and \eqref{uj-uniform-lower-bd} the equation \eqref{fast-diffusion-eqn} is uniformly parabolic for the sequence $\{u_j\}_{j=1}^{\infty}$ on every compact subset of $\1{\Omega}_{\delta}\times (0,\infty)$. 
By the Schauder estimates \cite{LSU} (Theorem 3.1 and Theorem 5.4 in chapter 5 of \cite{LSU}), for any compact subset $K$ of $\Omega_{\delta}\times (0,\infty)$,
\begin{equation}\label{1st-2nd-derivatives-infty-bd}
\sup_K(|\nabla u_j|+|\1^2_{x_kx_l}u_j|+|u_{j,t}|)\le C_3\quad\forall k,l=1,\cdots,n, j\in\Z^+
\end{equation}
and
\begin{equation}\label{2nd-derivatives-holder}
\sup_{(y,s), (y',s')\in  K}\frac{|\1^2_{x_kx_l}u_j(y,s)-\1^2_{x_kx_l}u_j(y',s')|}{|y-y'|^{\alpha}}+\sup_{(y,s), (y',s')\in  K}\frac{|u_{j,t}(y,s)-u_{j,t}(y',s')|}{|s-s'|^{\frac{\alpha}{2}}}\le C_4\quad\forall k,l=1,\cdots,n, j\in\Z^+
\end{equation}
for some constants $C_3>0$, $C_4>0$, $0<\alpha<1$.

Hence by the Schauder estimates \cite{LSU} the sequence $\{u_j\}_{j=1}^{\infty}$ is equi-H\"older continuous in $C^{2,1}(K)$ for any compact subset $K$ of 
$\Omega_{\delta}\times (0,T)$. Hence by the Ascoli Theorem and a diagonalization argument the sequence $\{u_j\}_{j=1}^{\infty}$ has a subsequence which we may assume without loss of generality to be the sequence itself that converges uniformly in $C^{2,1}(K)$ for any  compact subset  $K\subset\Omega_{\delta}\times(0,\infty)$ as $j\to\infty$ to some function $u\in C^{2,1}(\Omega_{\delta}\times(0,\infty))$ that satisfies
\begin{align}\label{weak-soln-eqn}
&\int_{t_1}^{t_2}\int_{\Omega_{\delta}}(u\eta_t+u^m\Delta\eta)\,dxdt+\int_{t_1}^{t_2}\int_{\1\Omega}f\eta\,d\sigma dt
+\sum_{i=1}^{i_0}\int_{t_1}^{t_2}\int_{\1 B_{\delta}(a_i)}g_i\eta\,d\sigma dt\notag\\
&\qquad \qquad =\int_{\Omega_{\delta}}u(x,t_2)\eta (x)\,dx-\int_{\Omega_{\delta}}u(x,t_1)\eta (x)\,dx
\end{align}
for any $t_2>t_1>0$, and $\eta\in C^2(\2{\Omega}_{\delta}\times (0,\infty))$ satisfying $\1\eta/\1\nu =0$ on 
$\1\Omega_{\delta}\times (0,\infty)$. 
Letting $j\to\infty$ in \eqref{eq-integration-form-from-main-construction-equation-1-with-j-2}, by \eqref{eq-cases-aligned-approximations-to-u-by-u-j-08}, \eqref{eq-cases-aligned-construction-of-g-i-j-converging-to-g-i-08}, \eqref{eq-cases-aligned-construction-of-f-i-converging-to-g-i-08} and the Lebesgue Dominated Convergence Theorem, we get that $u$ satisfies \eqref{mass-formulua-holes}. By Lemma \ref{positivity-lemma},
\begin{equation}\label{eq-strictly-positivity-of-limit-function-u-through-whole-domain-324}
u(x,t)>0\quad\forall\, (x,t)\in\Omega_{\delta}\times(0,\infty).
\end{equation}
By an argument same as case 1, $u$ satisfies \eqref{u-L1-converge-u0-as-t-goto-0}. Hence $u$ is a solution of \eqref{Neumann-holes-eqn}.\\
\noindent\textbf{Case 3:} $0\leq u_0\in L^p(\Omega_{\delta})$.\\
\noindent For any $j\in\Z^+$, let $u_{0,j}(x)=\min\left(u_0(x),j\right)$. By case 2 there exists a solution $u_j$ of \eqref{Neumann-holes-eqn} in $\Omega_{\delta}\times(0,T_j)$ with initial value $u_{0,j}$ that satisfies \eqref{mass-formulua-holes}. Since
$\|u_{0,j}\|_{L^{p}(\Omega_{\delta})}\le\|u_0\|_{L^{p}(\Omega_{\delta})}$, by  Proposition \ref{Lp-L-infty-thm-2} for any $t_2>t_1>0$ there exists a constant $C$ such that \eqref{uj-uniform-upper-bd} holds.
Since $u_{0,j}$ increases and converges to $u_0$ as $j\to\infty$, by Lemma \ref{comparsion-lemma1},
\begin{equation}\label{eq-increasing-sequence-of-u-sub-j-on-L-1-space03}
u_j(x,t)\le u_{j+1}(x,t)\quad\forall x\in\Omega_{\delta},0<t<T_j,j\in\Z^+.
\end{equation}
Since $u_j>0$ in $\Omega_{\delta}\times (0,\infty)$, by \eqref{uj-uniform-upper-bd} and \eqref{eq-increasing-sequence-of-u-sub-j-on-L-1-space03} the equation \eqref{fast-diffusion-eqn} for the sequence $\{u_j\}_{j=1}^{\infty}$ is uniformly parabolic on $\2{\Omega}_{\delta}\times [t_1,t_2]$ for any $t_2>t_1>0$. By the Schauder estimates \cite{LSU} (Theorem 3.1 and Theorem 5.4 in chapter 5 of \cite{LSU}), for any compact subset $K$ of $\Omega_{\delta}\times (0,\infty)$,
\eqref{1st-2nd-derivatives-infty-bd} and \eqref{2nd-derivatives-holder} hold
for some constants $C_3>0$, $C_4>0$, $0<\alpha<1$. 

Hence the sequence $\{u_j\}_{j=1}^{\infty}$ is uniformly equi-H\"older continuous in $C^{2,1}(K)$ for any compact set $K\subset\Omega_{\delta}\times (0,\infty)$.  Hence by \eqref{eq-increasing-sequence-of-u-sub-j-on-L-1-space03}, the Ascoli theorem, and a diagonalization argument and  the sequence $\{u_j\}_{j=1}^{\infty}$ increases and converges in $C^{2,1}(K)$ for any compact set $K\subset\Omega_{\delta}\times (0,\infty)$ to a solution $u\in C^{2,1}(\Omega_{\delta}\times (0,\infty))$ of \eqref{fast-diffusion-eqn} in $\Omega_{\delta}\times (0,T)$. Putting $u=u_j$ and $u_0=u_{0,j}$ in \eqref{mass-formulua-holes}, \eqref{uj-uniform-upper-bd}, \eqref{weak-soln-eqn}, and  letting $j\to\infty$, by \eqref{uj-uniform-upper-bd} we get that $u$ satisfies \eqref{mass-formulua-holes},\eqref{weak-soln-eqn} and $u\in L^{\infty}(\1{\Omega}_{\delta}\times (0,\infty))$. It remains to show that $u$ has initial value $u_0$.\\
\indent Let $t_k\to 0$ as $k\to\infty$. By the same argument as in the Case 1, the sequence $\left\{u(x,t_k)\right\}_{k=1}^{\infty}$ has a subsequence which we may assume without loss of generality to the sequence itself such that $u(x,t_k)\to u_0$ weakly in $L^1(\Omega_{\delta})$ and a.e. $x\in\Omega_{\delta}$ as $k\to\infty$. By the proof of  Proposition \ref{Lp-L-infty-thm-2}, \eqref{u^p-integral-bd} holds. Hence $u(x,t_{k})$ converges weakly in $L^p(\Omega_{\delta})$ to some function $v_0$ as 
$k\to\infty$. Then there exists a subsequence of $\{t_{k}\}_{k=1}^{\infty}$ which we may assume without loss of generality to be the sequence $\{t_{k}\}_{k=1}^{\infty}$ such that $u(x,t_{k})$ converges to $v_0(x)$ a.e. $x\in\Omega_{\delta}$ as $k\to\infty$. Hence $v(x)=u_0(x)$ a.e. $x\in\Omega_{\delta}$. Thus
\begin{equation}\label{u0<u-lp-norm-at-t=0}
\int_{\Omega_{\delta}}u_0^p\,dx\le\liminf_{k\to\infty}\int_{\Omega_{\delta}}u^p(x,t_{k})\,dx.
\end{equation}
Letting first $t\to 0$ and then $a\to 0$ in \eqref{u^p-integral-bd},
\begin{equation}\label{u0>u-lp-norm-at-t=0}
\limsup_{t\to 0}\int_{\Omega_{\delta}}u^p(x,t)\,dx\le\int_{\Omega_{\delta}}u_0^p\,dx.
\end{equation}
By \eqref{u0<u-lp-norm-at-t=0} and \eqref{u0>u-lp-norm-at-t=0},
\begin{equation}\label{u=u0-lp-norm-at-t=0}
\lim_{l\to 0}\int_{\Omega_{\delta}}u^p(x,t_{k})\,dx=\int_{\Omega_{\delta}}u_0^p\,dx.
\end{equation}
Now consider the function
\begin{equation*}
w_k(x)=2^p(u^p(x,t_{k})+u_0^p(x))-|u(x,t_{k})-u_0(x)|^p.
\end{equation*}
Note that $w_k(x)\ge 0$ on $\Omega_{\delta}$ and $w_k(x)\to 2^{p+1}u_0^p(x)$ a.e. $x\in\Omega_{\delta}$ as $k\to\infty$. Hence by the Fatou Lemma and
\eqref{u=u0-lp-norm-at-t=0},
\begin{align*}
2^{p+1}\int_{\Omega_{\delta}}u_0^p\,dx\le&\liminf_{k\to\infty}\int_{\Omega_{\delta}}2^p(u^p(x,t_{k})+u_0^p(x))-|u(x,t_{k})-u_0(x)|^p\,dx\\
=&2^{p+1}\int_{\Omega_{\delta}}u_0^p\,dx-\limsup_{k\to\infty}\int_{\Omega_{\delta}}|u(x,t_{k})-u_0(x)|^p\,dx\\
\Rightarrow\qquad\lim_{l\to\infty}\int_{\Omega_{\delta}}|u(x,t_{k})-&u_0(x)|^p\,dx=0.
\end{align*}
Since the sequence $\left\{t_k\right\}_{k=1}^{\infty}$ is arbitrary, $u$ satisfies \eqref{u-L1-converge-u0-as-t-goto-0}. Hence, $u$ is a  solution of \eqref{Neumann-holes-eqn} in $\Omega_{\delta}\times (0,\infty)$.\\ 
\indent If $f\equiv 0$ on $\1\Omega\times (0,\infty)$ and $g_i$, $i=1,\dots, i_0$, are positive monotone decreasing functions of 
$t>0$, then by the choice of the approximating functions for $f$ and $g_i$ in the construction of solutions of cases 1,2,3 above and Lemma \ref{positivity-lemma} $u$ satisfies \eqref{Aronson-Benilan-ineqn} in 
$\Omega_{\delta}\times(0,\infty)$ and the theorem follows.
\end{proof}

By the same argument as the proof of Theorem \ref{Neumann-holes-problem-existence-thm} and Lemma \ref{comparsion-lemma1} we have the following two results.

\begin{thm}\label{Neumann-problem-bded-domain-thm}
Let $n\ge 3$, $0<m\le\frac{n-2}{n}$, $0\le u_0\in L^p(\Omega)$ for some constant $p>\frac{n(1-m)}{2}$, and $0\le f\in L_{loc}^{\infty}(\partial\Omega\times [0,\infty))$. Suppose either $u_0\not\equiv 0$ on $\Omega_{\delta}$ or 
\begin{equation*}
\int_0^t\int_{\partial\Omega}f\,d\sigma ds>0\quad \forall t>0
\end{equation*}
holds. Then there exists a unique solution $u$ for the Neumann problem,
\begin{equation*}
\begin{cases}
\begin{aligned}
u_t=&\La u^m \quad\mbox{in $\Omega\times(0,\infty)$}\\
\frac{\partial u^m}{\partial\nu}=&f \qquad \mbox{ on }\partial\Omega\times(0,\infty)\\
u(x,0)=&u_0(x)\quad\mbox{in }\Omega
\end{aligned}
\end{cases}
\end{equation*}
which satisfies
\begin{equation*}
\int_{\Omega}u(x,t)\,dx=\int_{\Omega}u_0\,dx+\int_0^t\int_{\partial\Omega}f\,d\sigma ds\quad \forall t>0.
\end{equation*}
Moreover if $f\equiv 0$ on $\1\Omega\times (0,\infty)$, then $u$ satisfies \eqref{Aronson-Benilan-ineqn} in 
$\Omega_{\delta}\times(0,T)$.
\end{thm}

\begin{cor}\label{existence-radially-symmetric-soln-Neumann-with-1-hole-cor}
Let $n\ge 3$, $0<\delta<R$, $0<m\leq\frac{n-2}{n}$, $p>\frac{n(1-m)}{2}$ and let $f,g\in L_{loc}^{\infty}([0,\infty))$ be two nonnegative functions. Suppose $0\le u_0\in L^p(B_R\bs B_{\delta})$ is a radially symmetric function such that either $u_0\not\equiv 0$ on $\Omega_{\delta}$ or 
\begin{equation*}
\int_0^t\int_{\partial B_R}f\,d\sigma ds+\int_0^t\int_{\partial B_{\delta}}g\,d\sigma ds>0\quad \forall t>0
\end{equation*}
holds. Then there exists a unique solution of 
\begin{equation*}
\left\{\begin{aligned}
u_t=&\La u^m\quad\mbox{in }(B_R\bs B_{\delta})\times(0,\infty)\\
\frac{\partial u^m}{\partial\nu}=&f\qquad\mbox{ on }\partial B_R\times(0,\infty)\\
\frac{\partial u^m}{\partial\nu}=&g\qquad\mbox{ on }\partial B_{\delta}\times(0,\infty)\\
u(x,0)=&u_0(x)\quad\mbox{in }B_R
\end{aligned}\right.
\end{equation*}
in $(B_R\bs B_{\delta})\times(0,\infty)$ which is radially symmetric and satisfies \eqref{mass-formulua-holes} with $\Omega=B_{R}$, $i_0=1$ and $a_1=0$.
\end{cor}

\section{Existence of singular solutions}
\setcounter{equation}{0}
\setcounter{thm}{0}

In this section we will prove the existence of singular solutions of \eqref{fast-diffusion-eqn} in $\hat{\Omega}\times (0,T)$.

\begin{proof}[Proof of Theorem \ref{singular-soln-existence-thm}]

\noindent Let $\alpha=2m(q+4\delta_1^{-2})$ and let 
\begin{equation*}
0<\3_j<\min\left(\frac{\delta_1}{2},\frac{(1-m)^2q^2}{(4+(1-m)q)^2},\frac{(1-m)q\delta_0}{4+(1-m)q}\right)\quad\forall j\in\Z^+.
\end{equation*}
be a sequence  decreasing to zero as $j\to\infty$. By Theorem \ref{Neumann-holes-problem-existence-thm} for any $j\in\Z^+$ there 
exists a solution $u_j$ of 
\begin{equation}\label{Neumann-problem-in-punctured-ball}
\left\{\begin{aligned}
u_t=&\Delta u^m\qquad\mbox{in }\Omega_{\delta}\times(0,\infty)\\
\frac{\partial u^m}{\partial\nu}=&f\qquad\quad\mbox{ on }\partial\Omega\times(0,\infty)\\
\frac{\partial u^m}{\partial\nu}=&\frac{\alpha}{\3_j^{qm+1}}\quad\,\,\mbox{ on }\partial B_{\3_j}(a_i)\times(0,\infty)\quad\forall
i=1,\cdots, i_0\\
u(x,0)=&u_0(x)\qquad\mbox{in }B_R
\end{aligned}\right.
\end{equation}
By Lemma \ref{soln-bded-holes-lemma} there exists a constant $0<\delta_2<\delta_1$ such that for any $T>0$ there exists a constant $A_1>0$ such that 
\begin{equation}\label{uj-upper-bd-phi-tidle}
u_j(x,t)\le\phi_{A_1}(x-a_i,t) \qquad \forall\3_j\le |x-a_i|<\delta_1,\,\,0\leq t<T, i=1,\cdots,i_0,\3_j<\delta_2
\end{equation}
where $\phi_{A_1}$ is given by \eqref{phi-A-defn}. By Lemma \ref{comparison-lem2},
\begin{equation}\label{uj-lower-bd-at-blow-up-pt}
u_j(x,t)\ge\frac{C_1}{|x-a_i|^qe^{\frac{1}{\delta_1^2-|x-a_i|^2}}}\quad\forall \3_j\le |x-a_i|<\delta_1,t>0, \3_j\le\frac{\delta_1}{2}, i=1,\cdots,i_0
\end{equation}
holds for any $0<\delta\le\delta_1/2$. By Proposition \ref{Lp-L-infty-thm-3} for any $t_2>t_1>0$ and $\delta'<\delta_1$ there exists a constant $M_{\delta', t_1,t_2}>0$ such that
\begin{equation}\label{uj-upper-bd5}
u_j(x,t)\le M_{\delta', t_1,t_2}\quad\forall x\in\Omega_{\delta'}, t_1\le t\le t_2, \3_j<\frac{\delta'}{2}.
\end{equation}
By \eqref{uj-upper-bd5} and  Theorem 1.1 of \cite{S} the sequence $\{u_j\}_{\3_j<\frac{\delta'}{2}}$ is equi-H\"older
continuous on $\Omega_{\delta'}\times [t_1,t_2]$ for any $0<\delta'<\delta_1$ and $t_2>t_1>0$. By the Ascoli Theorem and a diagonalization argument the sequence $\{u_{j}\}_{j=1}^{\infty}$ has a subsequence which we may assume without loss of generality to be the sequence itself that converges uniformly on every compact subset of $\hat{\Omega}\times(0,\infty)$ 
to some continuous function $u$ that satisfies \eqref{fast-diffusion-eqn} in $\mathcal{D}(\hat{\Omega}\times (0,\infty))$ as $j\to\infty$. 
Letting $j\to\infty$ in \eqref{uj-lower-bd-at-blow-up-pt}, \eqref{uj-upper-bd-phi-tidle},  and \eqref{uj-upper-bd5},
we get \eqref{singular-soln-upper-lower-bd},
\begin{equation}\label{u-upper-bd-phi-tidle2}
u(x,t)\le\phi_{A_1}(x-a_i,t) \qquad \forall 0< |x-a_i|<\delta_1,\,\,0\leq t<T, i=1,\cdots,i_0
\end{equation}
and
\begin{equation}\label{u-upper-bd5}
u(x,t)\le M_{\delta', t_1,t_2}\quad\forall x\in\Omega_{\delta'}, t_1\le t\le t_2.
\end{equation}
By \eqref{u-upper-bd-phi-tidle2}, \eqref{singular-soln-upper-lower-bd2} follows.
By \eqref{singular-soln-upper-lower-bd} and Lemma \ref{positivity-lemma}, 
\begin{equation*}
u>0\quad\mbox{ in }\hat{\Omega}\times (0,\infty).
\end{equation*}
Hence for any $t_2>t_1>0$ and $0<\delta'<\delta_0$, there exists a constant $M'>0$ such that
\begin{equation}\label{u-lower-bd5}
u\ge M'\quad\mbox{ in }\Omega_{\delta'}\times (t_1,t_2).
\end{equation}
By \eqref{u-upper-bd5} and \eqref{u-lower-bd5} the equation \eqref{fast-diffusion-eqn} for $u$ is uniformly parabolic on
$\Omega_{\delta'}\times (t_1,t_2)$ for any $t_2>t_1>0$ and $0<\delta'<\delta_0$. 
Hence by the Schauder estimates \cite{LSU} (Theorem 3.1 and Theorem 5.4 in chapter 5 of \cite{LSU}), $u\in C^{2+\beta,1+(\beta/2)}(\hat{\Omega}\times (0,\infty))$ for some constant $0<\beta<1$ is a classical solution of \eqref{fast-diffusion-eqn} in $\hat{\Omega}\times (0,\infty)$. Putting $u=u_j$ in \eqref{weak-singular-soln-defn} and letting $j\to\infty$, we get that $u$ satisfies \eqref{weak-singular-soln-defn} for any $0<t_1<t_2<T$ and $\eta\in C_0^2((\overline{\Omega}\setminus\{a_1,\dots,a_{i_0}\})\times(0,T))$ satisfying $\1\eta/\1\nu =0$ on $\partial\Omega\times (0,T)$. 

Let $\eta\in C_0^{\infty}(\hat{\Omega})$ and $K=\mbox{supp}\,\eta$. Then $\3=\mbox{dist}(K,\1\Omega\cup\{0\})>0$. Let $\psi\in C_0^{\infty}(\hat{\Omega})$, $0\le\psi\le 1$, be such that $\psi (x)=1$ for all $x\in K$ and $\psi (x)=0$ for all $\mbox{dist}(x,K)\ge\3/2$. 
Let $K_1=\{x\in\Omega:\mbox{dist}(x,K)\le\3/2\}$. By the proof of Lemma 3.1 of \cite{HP} there exist constants $\alpha>1$ and $C>0$ such that
\begin{align}\label{uj-L1-uniform-bd}
&\left|\frac{\1}{\1 t}\left(\int_{\Omega_{\3_j}}u_j\psi^{\alpha}\,dx\right)\right|\le C\left(\int_{\Omega_{\3_j}}u_j\psi^{\alpha}\,dx\right)^m\quad\forall t>0,\3_j<\3/2\notag\\ 
\Rightarrow\quad&\left(\int_{\Omega_{\3_j}}u_j(x,t)\psi^{\alpha}(x)\,dx\right)^{1-m}\le\left(\int_{\Omega_{\3_j}}u_0\psi^{\alpha}\,dx\right)^{1-m}+(1-m)Ct\qquad \forall t>0,\3_j<\3/2\notag\\
\Rightarrow\quad&\int_Ku_j(x,t)\,dx\le c_{K_1}\qquad \forall 0<t<1,\3_j<\3/2.
\end{align}
where
$$
c_{K_1}=\left(\left(\int_{K_1}u_0\,dx\right)^{1-m}+(1-m)C\right)^{\frac{1}{1-m}}
$$
Hence by \eqref{uj-L1-uniform-bd},
\begin{align}\label{uj-u0-L1-difference}
\left|\int_{\Omega_{\3_j}}u_j\eta\,dx-\int_{\Omega_{\3_j}}u_0\eta\,dx\right|\le&\int_0^t\int_{\Omega_{\3_j}}u_j^m|\Delta\eta|\,dx\,dt\notag\\
\le&\|\Delta\eta\|_{L^{\infty}}|K|^{1-m}t\left(\int_Ku_j(x,t)\,dx\right)^m\notag\\
\le&\|\Delta\eta\|_{L^{\infty}}|K|^{1-m}c_{K_1}^mt\quad\forall 0<t<1,\3_j<\3/2.
\end{align}
Letting $j\to\infty$ in \eqref{uj-u0-L1-difference},
\begin{equation}\label{u-u0-L1-difference}
\left|\int_{\hat{\Omega}}u\eta\,dx-\int_{\hat{\Omega}}u_0\eta\,dx\right|\le\|\Delta\eta\|_{L^{\infty}}|K|^{1-m}c_{K_1}^mt\quad\forall 0<t<1.
\end{equation}
Letting $t\to 0$ in \eqref{u-u0-L1-difference} we get \eqref{u-L1-converge-u0-as-t-goto-0-Omega-hat} and Theorem \ref{singular-soln-existence-thm} follows.
\end{proof}

We are now ready for the proof of Theorem \ref{singular-soln-existence-thm2}. 

\begin{proof}[Proof of Theorem \ref{singular-soln-existence-thm2}]
Let $R_0>3\max_{1\le i\le i_0}|a_i|$. Then for any integer $j>R_0$ by Theorem \ref{singular-soln-existence-thm}  there exists a solution $u_j$ of 
\begin{equation*}
\left\{\begin{aligned}
u_t=&\Delta u^m\quad\mbox{ in }\hat{B}_j\times(0,\infty)\\
\frac{\partial u^m}{\partial\nu}=&0\qquad\, \mbox{ on }\partial B_j\times(0,\infty)\\
u(x,0)=&u_0(x)\quad\mbox{ in }\hat{B}_j
\end{aligned}\right.
\end{equation*}
which satisfies 
\begin{equation}\label{uj-uniform-lower-blow-up-rate}
u_j(x,t)\ge\frac{C_1}{|x-a_i|^qe^{\frac{1}{\delta_1^2-|x-a_i|^2}}}\quad\forall 0<|x-a_i|<\delta_1,t>0, i=1,\cdots,i_0
\end{equation}
and for any $T>0$ there exists a constant $A_1>0$ such that 
\begin{equation}\label{uj-uniform-upper-blow-up-rate}
u_j(x,t)\le\phi_{A_1}(x-a_i,t) \qquad \forall 0< |x-a_i|<\delta_1,\,\,0\leq t<T, i=1,\cdots,i_0,
\end{equation}
where $\phi_{A_1}$ is given by \eqref{phi-A-defn}. By \eqref{uj-uniform-lower-blow-up-rate}, \eqref{uj-uniform-upper-blow-up-rate} and the same argument as the proof of Theorem \ref{singular-soln-existence-thm}
the sequence $\{u_j\}_{j>R_0}$ has a subsequence that converges to a $C^{2,1}$ solution $u$ of \eqref{singular-Neumann-problem2} that
satisfies \eqref{singular-soln-upper-lower-bd} and \eqref{singular-soln-upper-lower-bd2} for some constant $C_T>0$. This completes the proof of Theorem \ref{singular-soln-existence-thm2}.

\end{proof}

\end{document}